\documentclass[11pt,oneside,english]{amsart}
\usepackage[T1]{fontenc}
\usepackage[utf8]{inputenc}
\usepackage[skip=\medskipamount]{parskip}
\usepackage{amstext}
\usepackage{amsthm}
\usepackage{amssymb}
\usepackage{setspace}
\makeatletter
\numberwithin{equation}{section}
\numberwithin{figure}{section}
\theoremstyle{plain}
\newtheorem{thm}{\protect\theoremname}[section]
\theoremstyle{definition}
\newtheorem{defn}[thm]{\protect\definitionname}
\theoremstyle{definition}
\newtheorem{xca}[thm]{\protect\exercisename}
\theoremstyle{plain}
\newtheorem{prop}[thm]{\protect\propositionname}
\theoremstyle{plain}
\newtheorem{lem}[thm]{\protect\lemmaname}
\theoremstyle{plain}
\newtheorem{cor}[thm]{\protect\corollaryname}
\theoremstyle{remark}
\newtheorem{rem}[thm]{\protect\remarkname}

\usepackage{ae,aecompl} 
\usepackage{fullpage}

\usepackage{tikz}
\usepackage{pgfplots}
\usepgfplotslibrary{colormaps,fillbetween}
\usetikzlibrary{3d}

\makeatother

\usepackage{babel}
\providecommand{\corollaryname}{Corollary}
\providecommand{\definitionname}{Definition}
\providecommand{\exercisename}{Exercise}
\providecommand{\lemmaname}{Lemma}
\providecommand{\propositionname}{Proposition}
\providecommand{\remarkname}{Remark}
\providecommand{\theoremname}{Theorem}

\begin{document}
\global\long\def\RR{\mathbb{R}}%

\global\long\def\SS{\mathbb{S}}%

\global\long\def\dd{\mathrm{d}}%

\global\long\def\HH{\mathcal{H}}%

\global\long\def\oo{\mathbf{1}}%

\global\long\def\inte{\operatorname{int}}%

\global\long\def\epi{\operatorname{epi}}%

\global\long\def\dom{\operatorname{dom}}%

\global\long\def\supp{\operatorname{supp}}%

\global\long\def\proj{\operatorname{Proj}}%

\global\long\def\Id{\mathrm{Id}}%

\global\long\def\kk{\mathcal{K}^{n}}%

\global\long\def\cvx{\mathrm{Cvx}_{n}}%

\global\long\def\lc{\mathrm{LC}_{n}}%

\global\long\def\conv{\operatorname{conv}}%

\title{On the functional Minkowski problem}
\author{Tomer Falah}
\address{Technion Israel Institute of Technology, Technion City, Haifa 3200003,
Israel}
\email{fatomer@campus.technion.ac.il}
\author{Liran Rotem}
\address{Technion Israel Institute of Technology, Technion City, Haifa 3200003,
Israel}
\email{lrotem@technion.ac.il}
\begin{abstract}
To every log-concave function $f$ one may associate a pair of measures
$(\mu_{f},\nu_{f})$ which are the surface area measures of $f$.
These are a functional extension of the classical surface area measure
of a convex body, and measure how the integral $\int f$ changes under
perturbations. The functional Minkowski problem then asks which pairs
of measures can be obtained as the surface area measures of a log-concave
function. In this work we fully solve this problem. 

Furthermore, we prove that the surface area measures are continuous
in correct topology: If $f_{k}\to f$, then $\left(\mu_{f_{k}},\nu_{f_{k}}\right)\to\left(\mu_{f},\nu_{f}\right)$
in the appropriate sense. Finding the appropriate mode of convergence
of the pairs $\left(\mu_{f_{k}},\nu_{f_{k}}\right)$ sheds a new light
on the construction of functional surface area measures. To prove
this continuity theorem we associate to every convex function a new
type of radial function, which seems to be an interesting construction
on its own right. 

Finally, we prove that the solution to functional Minkowski problem
is continuous in the data, in the sense that if $\left(\mu_{f_{k}},\nu_{f_{k}}\right)\to\left(\mu_{f},\nu_{f}\right)$
then $f_{k}\to f$ up to translations. 
\end{abstract}

\maketitle

\section{Introduction}

\subsection{Surface area measures and the Minkowski problem}

The Minkowski problem is one of the most famous problems in convex
geometry and differential geometry. We begin with recalling some classical
results about it, and about surface area measures in general. For
more general background on convex geometry the reader may consult
e.g. \cite{SchneiderConvex2013}. 

The Minkowski sum of two closed convex sets $K,L\subseteq\RR^{n}$
is defined as 
\[
K+L=\left\{ x+y:\ x\in K,\ y\in L\right\} .
\]
The set $K+L$ is convex, and is closed if $K$ or $L$ are bounded.
If in addition $\lambda>0$, we define the dilation $\lambda K=\left\{ \lambda x:\ x\in K\right\} $.
The support function $h_{K}:\SS^{n-1}\to(-\infty,\infty]$ of a closed
convex set $K$ is given by $h_{K}(y)=\sup_{x\in K}\left\langle x,y\right\rangle $.
These constructions interact well with each other, in the sense that
$\lambda K+\mu K=(\lambda+\mu)K$ and $h_{\lambda K+\mu L}=\lambda h_{K}+\mu h_{L}$.
The (Lebesgue) volume of $K$ will be denoted by $\left|K\right|$. 

By a convex body we mean a compact convex set with non-empty interior,
or equivalently a closed convex set $K\subseteq\RR^{n}$ such that
$0<\left|K\right|<\infty$. For every convex body $K$ there exists
a finite Borel measure $S_{K}$ on the unit sphere $\SS^{n-1}=\left\{ x\in\RR^{n}:\ \left|x\right|=1\right\} $
with the following property: For every closed convex set $L\subseteq\RR^{n}$
we have
\[
\lim_{t\to0^{+}}\frac{\left|K+tL\right|-\left|K\right|}{t}=\int_{\SS^{n-1}}h_{L}\dd S_{K}.
\]
 The measure $S_{K}$ is called the \emph{surface area measure} of
$K$. 

The surface area measure $S_{K}$ can also be defined explicitly.
Let $n_{K}:\partial K\to\SS^{n-1}$ be the Gauss map, associating
to every boundary point $x\in\partial K$ the outer unit normal to
$K$ at $x$. It is known that this outer unit normal is unique $\HH^{n-1}$-almost
everywhere, where $\HH^{n-1}$ denotes the $(n-1$)-dimensional Hausdorff
measure. We then have $S_{K}=\left(n_{K}\right)_{\sharp}\left(\left.\HH^{n-1}\right|_{\partial K}\right)$,
where $\sharp$ denotes the push-forward of a measure. Explicitly,
this means that for every (say continuous) function $\rho:\SS^{n-1}\to\RR$
we have 
\[
\int_{\SS^{n-1}}\rho\dd S_{K}=\int_{\partial K}\left(\rho\circ n_{K}\right)\dd\HH^{n-1}.
\]
 For example, when $K$ is a polytope the measure $S_{K}$ is discrete,
with atoms corresponding to the outer normals to the facets of $K$.
When $\partial K$ is smooth with non-vanishing Gauss curvature the
measure $S_{K}$ is continuous with respect to the $\left.\HH^{n-1}\right|_{\SS^{n-1}}$,
and $\frac{\dd S_{K}}{\dd\HH^{n-1}}(\theta)=\det D^{2}h_{K}(\theta)$,
which is also the inverse of the Gauss curvature of $\partial K$
at the point $n_{K}^{-1}(\theta)$. Here $D^{2}h_{K}=\nabla_{\SS^{n-1}}^{2}h_{K}+h_{K}\cdot\Id$,
where $\nabla_{\SS^{n-1}}^{2}$ is the Riemannian Hessian on $\SS^{n-1}$. 

The Minkowski problem asks which measures on $\SS^{n-1}$ are the
surface area measures of a convex body. Its solution, known as Minkowski's
existence theorem, reads as follows:
\begin{thm}
\label{thm:minkowski}Let $\mu$ be a finite Borel measure on $\SS^{n-1}$.
Then $\mu=S_{K}$ for some convex body $K\subseteq\RR^{n}$ if and
only if $\mu$ satisfies the following two conditions:
\begin{enumerate}
\item $\mu$ is centered, i.e. $\int_{\SS^{n-1}}\left\langle x,\theta\right\rangle \dd\mu(x)=0$
for all $\theta\in\RR^{n}$. 
\item $\mu$ is not supported on any hyperplane.
\end{enumerate}
Moreover, the body $K$ is unique up to translations: If $S_{K}=S_{L}$
then $K=L+v$ for some $v\in\RR^{n}$. 
\end{thm}

We refer the reader to \cite{SchneiderConvex2013} for two proofs
of Minkowski's theorem. One is Minkowski's original argument for polytopes,
given in Theorem 8.2.1, followed by an approximation argument of Fenchel
and Jessen in Theorem 8.2.2. An alternative proof due to Alexandrov
appears (for a generalized problem) in Theorem 9.2.1. This latter
proof is the one that serves as a motivation for our proof of Theorem
\ref{thm:main-Minkowksi}. The reader may also be interested in the
notes following Section 8.2 for a more comprehensive history of the
problem, including many references. 

When the body $K$ is smooth and $\mu$ has smooth density $f$, the
equation $S_{K}=\mu$ is equivalent to the Monge--Ampère type partial
differential equation
\[
\det\left(D^{2}h_{K}\right)=f.
\]
 Theorem (\ref{thm:minkowski}) is then a statement of existence and
uniqueness of solutions to this equation, at least in a weak sense. 

On the space of convex bodies we have a natural topology, given by
the Hausdorff distance 
\begin{align*}
d_{H}(K,L) & =\inf\left\{ r>0:\ K\subseteq L+\overline{B}_{r}(0)\text{ and }L\subseteq K+\overline{B}_{r}(0)\right\} \\
 & =\max_{\theta\in\SS^{n-1}}\left|h_{K}(\theta)-h_{L}(\theta)\right|,
\end{align*}
 where $\overline{B}_{r}(0)$ denotes the closed ball of radius $r$
around the origin. The surface area measure is weakly continuous in
this topology: If $K_{i}\to K$ then $S_{K_{i}}\to S_{K}$ weakly
(see e.g. \cite[Theorem 4.2.1]{SchneiderConvex2013} for a much more
general result). It turns out that the converse to this statement
is also true, i.e. the solution $K$ to the Minkowski problem is continuous
in the data $\mu$: 
\begin{thm}
\label{thm:continuity-bodies}Let $\left\{ K_{i}\right\} _{i=1}^{\infty},K\subseteq\RR^{n}$
be convex bodies such that $S_{K_{i}}\to S_{K}$ weakly. Then there
exists translations $\left\{ v_{i}\right\} _{i=1}^{\infty}$ such
that $K_{i}+v_{i}\to K$ in the Hausdorff topology. 
\end{thm}

Theorem \ref{thm:continuity-bodies} appears to be folklore. Several
papers that prove extensions of this result refer to it simply as
``well-known'' (see e.g. \cite{ZhuContinuity2016}, \cite{WangEtAlContinuity2018},
\cite{WangEtAlContinuity2019}), and we were unable trace its origins.
In any case, we are only stating Theorem \ref{thm:continuity-bodies}
as a motivation and will not need to apply it at any point. A proof
can be easily constructed by adapting our proof of Theorem \ref{thm:functions-cont}.

\subsection{Functional surface area measures}

In this work we extend the results above from convex bodies to log-concave
functions. Recall that a function $f:\RR^{n}\to[0,\infty)$ is called
log-concave if for all $x,y\in\RR^{n}$ and $0\le\lambda\le1$ we
have 
\[
f\left((1-\lambda)x+\lambda y\right)\ge f(x)^{1-\lambda}f(y)^{\lambda}.
\]
 In other words, $f$ is log-concave if $f=e^{-\phi}$ for a convex
function $\phi:\RR^{n}\to(-\infty,\infty]$. We will always assume
our convex functions are lower semicontinuous and not identically
$+\infty$, or equivalently that our log-concave functions are upper
semicontinuous and not identically $0$. For every closed convex set
$K\subseteq\RR^{n}$ its indicator $\oo_{K}$ is a log-concave function,
and we would like to consider log-concave functions as ``generalized
convex bodies''. 

This approach, called ``geometrization of probability'' by V. Milman
(see \cite{MilmanGeometrization2008}), proved to be very useful in
convexity over the last two decades. Several open problems in convex
geometry, which seemed completely intractable using purely geometric
tools, were solved by extending the problem to the functional setting
and using analytic and probabilistic tools that only make sense when
dealing with general log-concave functions. A survey of such results
will take us too far from our goal, so let us only mention the recent
solution of Bourgain's slicing problem by Klartag and Lehec (\cite{KlartagLehecAffirmative2024}),
that would have been completely impossible using purely geometric
tools. Surface area measures are related to some of the most important
open problems in convex geometry such as the log-Brunn--Minkowski
problem (see \cite{BoroczkyEtAlLogBrunnMinkowski2012}), and it seems
natural that in order to attack such problem we need to develop the
theory of functional surface area measures. 

We define 
\[
\cvx=\left\{ \phi:\RR^{n}\to(-\infty,\infty]:\ \begin{array}{l}
\phi\text{ is convex, lower semicontinuous,}\\
\text{and }0<\int e^{-\phi}<\infty
\end{array}\right\} ,
\]
and $\lc=\left\{ e^{-\phi}:\ \phi\in\cvx\right\} $. This will be
our extension of the class of convex bodies, with the integral $\int f$
replacing the volume $\left|K\right|$. Note that if $0<\int e^{-\phi}<\infty$
then $\phi$ is \emph{coercive}, i.e. $\phi(x)\ge a\left|x\right|+b$
for some $a>0$ and $b\in\RR$ (see e.g. \cite[Lemma 2.5]{ColesantiFragalaFirst2013}).
The sup-convolution of two log-concave functions $f,g:\RR^{n}\to[0,\infty)$
is defined by 
\[
\left(f\star g\right)(x)=\sup_{y\in\RR^{n}}\left(f(y)g(x-y)\right).
\]
 This function is again log-concave, and is upper semi-continuous
assuming $f$ and $g$ are and at least one of them belongs to $\lc$.
For $\lambda>0$ we define the dilation $\lambda\cdot f$ by $\left(\lambda\cdot f\right)(x)=f\left(\frac{x}{\lambda}\right)^{\lambda}$.
Finally, the support function of $f$ is given by $h_{f}=\left(-\log f\right)^{\ast}$,
where 
\[
\psi^{\ast}(y)=\sup_{x\in\RR^{n}}\left(\left\langle x,y\right\rangle -\psi(x)\right)
\]
 denotes the Legendre transform. There are several good characterization
theorems that explain why these are the natural operations on the
class of log-concave functions -- see e.g. \cite{Artstein-AvidanMilmanConcept2009}
and \cite{RotemSupport2013}. For now we just mention that $\left(\lambda\cdot f\right)\star\left(\mu\cdot f\right)=(\lambda+\mu)\cdot f$
and that $h_{\left(\lambda\cdot f\right)\star\left(\mu\cdot g\right)}=\lambda h_{f}+\mu h_{g}$. 

Given these basic operations we can now consider the same first variation
as before:
\begin{defn}
For $f\in\lc$ and $g:\RR^{n}\to\RR$ an upper semicontinuous log-concave
function (not necessarily with $0<\int g<\infty$) we set 
\[
\delta(f,g)=\lim_{t\to0^{+}}\frac{\int\left(f\star\left(t\cdot g\right)\right)-\int f}{t}.
\]
 
\end{defn}

\begin{thm}
\label{thm:representation}For all $f=e^{-\phi}\in\lc$ and every
upper semicontinuous log-concave function $g$ we have 
\begin{equation}
\delta(f,g)=\int_{\RR^{n}}h_{g}\dd\mu_{f}+\int_{\SS^{n-1}}h_{\supp(g)}\dd\nu_{f}.\label{eq:representation}
\end{equation}
 Here $\mu_{f}$ is the Borel measure on $\RR^{n}$ defined by $\mu_{f}=\left(\nabla\phi\right)_{\sharp}\left(f\dd x\right)$
and $\nu_{f}$ is the Borel measure on $\SS^{n-1}$ defined by $\nu_{f}=\left(n_{\supp(f)}\right)_{\sharp}\left(f\dd\!\left.\HH^{n-1}\right|_{\partial\supp(f)}\right)$.
We refer to the pair $(\mu_{f},\nu_{f})$ as the \emph{surface area
measures} of $f$. 
\end{thm}

A few comments are in order. First, the measures $\mu_{f}$ and $\nu_{f}$
are always well defined for all $f\in\lc$. For $\mu_{f}$, this is
because $\phi$ is a convex function and therefore differentiable
almost everywhere on the set $\left\{ x:\ \phi(x)<\infty\right\} =\left\{ x:\ f(x)>0\right\} $.
For $\nu_{f}$, note that the support 
\[
\supp(f)=\overline{\left\{ x\in\RR^{n}:\ f(x)>0\right\} }
\]
 is a closed convex set with non-empty interior (that may be unbounded).
Therefore the Gauss map $n_{\supp(f)}:\partial\supp(f)\to\SS^{n-1}$
is well-defined $\HH^{n-1}$-almost everywhere. Of course we may have
$\supp(f)=\RR^{n}$, and then $\partial\supp(f)=\emptyset$; In this
case we set $\nu_{f}\equiv0$. 

Regarding the history of Theorem \ref{thm:representation}, the first
variation $\delta(f,g)$ was first considered by Klartag and Milman
in \cite{KlartagMilmanGeometry2005} in the special case $f(x)=e^{-\left|x\right|^{2}/2}$.
This notion was further studied in \cite{RotemMean2012}, where (\ref{eq:representation})
was proved for this special choice of a function $f$. At the same
time and independently, Colesanti and Fragalà studied the first variation
$\delta(f,g)$ for general log-concave functions (\cite{ColesantiFragalaFirst2013}).
Among other results, they proved formula (\ref{eq:representation})
under some strong smoothness and regularity assumptions on the functions
$f$ and $g$. Cordero-Erausquin and Klartag (\cite{Cordero-ErausquinKlartagMoment2015})
studied the measure $\mu_{f}$ under the name ``the moment measure
of $\phi=-\log f$''. They concentrated on the case $\nu_{f}\equiv0$,
and while they did not prove (\ref{eq:representation}) in this case
they proved a result similar in spirit. Note that $\nu_{f}\equiv0$
if and only if $f\equiv0$ $\HH^{n-1}$-almost everywhere on $\partial\supp(f)$,
which is equivalent to saying that $f$ is continuous $\HH^{n-1}$-almost
everywhere. Therefore \cite{Cordero-ErausquinKlartagMoment2015} named
this condition \emph{essential continuity}. Finally, formula (\ref{eq:representation})
was proven in full generality by the second named author (\cite{RotemSurface2022},
\cite{RotemAnisotropic2023}). 

Finally, we caution the reader that the support function appears in
formula (\ref{eq:representation}) in two different senses. The function
$h_{g}:\RR^{n}\to(-\infty,\infty]$ is the support function in the
sense of log-concave functions, i.e. $h_{g}=(-\log g)^{\ast}$. However,
$\supp(g)$ is a closed convex set, and its support function $h_{\supp(g)}:\SS^{n-1}\to\RR$
is in the classical meaning. These two support functions are related
though: the \emph{horizon function} $\overline{\psi}:\SS^{n-1}\to(-\infty,\infty]$
of a convex function $\psi:\RR^{n}\to(-\infty,\infty]$ is defined
as 
\begin{equation}
\overline{\psi}(\theta)=\lim_{\lambda\to\infty}\frac{\psi(p+\lambda\theta)}{\lambda},\label{eq:horizon-def}
\end{equation}
 where $p$ is an arbitrary point such that $\psi(p)<\infty$. Then
$\overline{\psi}$ is well-defined and independent of $p$ (see \cite[Theorem 3.21]{RockafellarWetsVariational1998}).
We then have $h_{\supp(g)}=\overline{h_{g}}$ (\cite[Theorem 11.5]{RockafellarWetsVariational1998}). 

\subsection{Our main theorems}

We can now state our first main theorem, a complete solution of the
Minkowski problem for functional surface area measures:
\begin{thm}
\label{thm:main-Minkowksi}Let $\mu$ be a finite Borel measure on
$\RR^{n}$ and $\nu$ be a finite Borel measure on $\SS^{n-1}$. Then
there exists a log-concave function $f\in\lc$ such that $(\mu_{f},\nu_{f})=(\mu,\nu)$
if and only if the pair $(\mu,\nu)$ satisfies the following conditions: 
\begin{enumerate}
\item $\mu$ is not identically $0$. 
\item $\mu$ has finite first moment, and the measure $\mu+\nu$ is centered
in the sense that for every $\theta\in\SS^{n-1}$ we have 
\[
\int_{\RR^{n}}\left\langle x,\theta\right\rangle \dd\mu(x)+\int_{\SS^{n-1}}\left\langle x,\theta\right\rangle \dd\nu(x)=0.
\]
\item $\mu$ and $\nu$ are not supported on a common hyperplane $H\subseteq\RR^{n}$. 
\end{enumerate}
In this case, the function $f$ is unique up to translations: If $g\in\lc$
also satisfies $\left(\mu_{g},\nu_{g}\right)=\left(\mu,\nu\right)$
then there exists $v\in\RR^{n}$ such that $g(x)=f(x+v)$. 
\end{thm}

The uniqueness part of Theorem \ref{thm:main-Minkowksi} was already
known -- it was proved in \cite{ColesantiFragalaFirst2013} under
regularity assumptions on $f$ and $g$, in \cite{Cordero-ErausquinKlartagMoment2015}
in the essentially continuous case, and in \cite{RotemAnisotropic2023}
for general $f,g\in\lc$. As for the existence part, when $\nu\equiv0$
Theorem \ref{thm:main-Minkowksi} was proved by Cordero-Erausquin
and Klartag (\cite{Cordero-ErausquinKlartagMoment2015}), with another
proof given by Santambrogio (\cite{SantambrogioDealing2016}). Variants
of the functional Minkowski problem, changing either the volume functional
or the addition operation, were studied by Huang, Liu, Xi and Zhao
(\cite{HuangEtAlDual2024}), Fang, Ye, Zhang, Zhao (\cite{FangEtAlDual2025}),
Fang, Ye, Zhang (\cite{FangEtAlRiesz2024}) and by the second named
author (\cite{RotemSurface2022}). However, to the best of our knowledge
all known results were partial results, in the sense that they deal
with the essentially continuous case $\nu\equiv0$. Dealing with general
log-concave functions does require new ideas, as we'll see below.

The functional Minkowski problem may be written as a pair of partial
differential equations. Indeed, assume $\mu$ has density $\rho_{1}:\RR^{n}\to\RR$
and $\nu$ has density $\rho_{2}:\SS^{n-1}\to\RR$. If $f=e^{-\phi}\in\lc$
is a smooth function on its smooth support $K=\supp(f)$, then the
change of variables formula shows that $\phi$ and $K$ solve the
system 
\begin{equation}
\begin{cases}
\rho_{1}\left(\nabla\phi(x)\right)\det\left(\left(\nabla^{2}\phi\right)(x)\right)=e^{-\phi(x)} & \text{for all }x\in\inte\left(K\right)\\
\frac{\rho_{2}\left(n_{K}(x)\right)}{\det\left(\left(D^{2}h_{K}\right)(n_{K}(x))\right)}=e^{-\phi(x)} & \text{for all }x\in\partial K.
\end{cases}\label{eq:minkowksi-pde}
\end{equation}
 Note that the domain $K$ is not given, but both $K$ and $\phi$
are the unknowns in this system of equations. Also note that $\det\left(\left(D^{2}h_{K}\right)(n_{K}(x))\right)^{-1}$
is simply the Gauss curvature of $K$ at the point $x\in\partial K$.
Theorem \ref{thm:main-Minkowksi} then states that under our assumptions
on $\rho_{1}$ and $\rho_{2}$ the system (\ref{eq:minkowksi-pde})
has a unique solution, at least in a weak sense. Studying the regularity
of our solution and whether it is also a classical solution to (\ref{eq:minkowksi-pde})
is an interesting problem beyond the scope of this paper. We refer
the reader to \cite{KlartagLogarithmicallyConcave2014} for some remarks
and references regarding regularity in the case $\nu\equiv0$ and
to \cite{UlivelliEntire2023} for a related result regarding the regularity
of the weighted Minkowski problem. 

Next, we discuss continuity of the surface area measures. On the class
$\cvx$ we have a natural notion of convergence called epi-convergence
-- see Definition \ref{def:epi-convergence} below and the discussion
following it. This gives us a topology on $\lc$ in the obvious way.
In \cite{KlartagLogarithmicallyConcave2014}, Klartag essentially
proved that if $f_{k}\to f$ then $\mu_{f_{k}}\to\mu_{f}$ weakly
(He claimed a weaker result, which does not use the notion of epi-convergence,
but the same ideas can be used to prove the more general statement).
However, the same cannot be true for the boundary measure $\nu_{f}$,
as a simple example shows:
\begin{xca}
Define $\left\{ \phi_{k}\right\} _{k=1}^{\infty}:\RR\to(-\infty,\infty]$
by $\phi_{k}(x)=\max\left(k\left|x\right|-k,0\right)$. Then $\phi_{k}\to\phi$,
where 
\[
\phi(x)=\oo_{[-1,1]}^{\infty}(x)=\begin{cases}
0 & x\in[-1,1]\\
+\infty & \text{otherwise.}
\end{cases}
\]

However $\nu_{e^{-\phi_{k}}}=0$ and $\nu_{e^{-\phi}}=\delta_{1}+\delta_{-1}\ne0$,
so clearly we cannot have $\nu_{e^{-\phi_{k}}}\to\nu_{e^{-\phi}}$
in any reasonable sense. 
\end{xca}

Indeed, it makes no sense to treat the measures $\mu_{f}$ and $\nu_{f}$
separately. Instead, the correct claim should be that if $f_{k}\to f$
then $\left(\mu_{f_{k}},\nu_{f_{k}}\right)\to\left(\mu,\nu\right)$
as a pair. We therefore make the following definitions: 
\begin{defn}
\label{def:cosmically-continuous}We say a function $\xi:\RR^{n}\to\RR$
is \emph{cosmically continuous} if:
\begin{enumerate}
\item $\xi$ is continuous in the usual sense on $\RR^{n}$.
\item The limit $\overline{\xi}(\theta)=\lim_{\lambda\to\infty}\frac{\xi(\lambda\theta)}{\lambda}$
exists (in the finite sense) \emph{uniformly} in $\theta\in\SS^{n-1}$. 
\end{enumerate}
\end{defn}

\begin{defn}
\label{def:cosmic-conv}Let $\left\{ \mu_{k}\right\} _{k=1}^{\infty}$
and $\mu$ be finite Borel measures on $\RR^{n}$, and let $\left\{ \nu_{k}\right\} _{k=1}^{\infty}$
and $\nu$ be finite Borel measures on $\SS^{n-1}$. We say that $(\mu_{k},\nu_{k})\to(\mu,\nu)$
\emph{cosmically} if for every cosmically continuous function $\xi:\RR^{n}\to\RR$
we have 
\[
\int_{\RR^{n}}\xi\dd\mu_{k}+\int_{\SS^{n-1}}\overline{\xi}\dd\nu_{k}\xrightarrow{k\to\infty}\int_{\RR^{n}}\xi\dd\mu+\int_{\SS^{n-1}}\overline{\xi}\dd\nu.
\]
\end{defn}

A similar class of functions appeared in the work of Ulivelli (\cite{UlivelliFirst2024})
under the notation $C_{rec}\left(\RR^{n}\right)$, for similar reasons.
However, functions in $C_{rec}\left(\RR^{n}\right)$ satisfy the stronger
assumption that $\left|\xi-\overline{\xi}\right|$ is bounded, when
$\overline{\xi}$ is considered as a $1$-homogeneous function on
$\RR^{n}$. We explain in Section \ref{sec:cosmic} why our definitions
appear to be the natural ones, as well as the origin of the name ``cosmic'':
The pair of measures $(\mu,\nu)$ can be naturally identified with
a single measure $\lambda$ on a compactification of $\RR^{n}$ known
as its cosmic closure (\cite[Chapter 3A]{RockafellarWetsVariational1998}).
Cosmic convergence $\left(\mu_{k},\nu_{k}\right)\to\left(\mu,\nu\right)$
is then nothing more than standard weak convergence $\lambda_{k}\to\lambda$. 

With the correct definition in place, we can state our second main
theorem: 
\begin{thm}
\label{thm:measures-cont}Fix $\left\{ f_{k}\right\} _{k=1}^{\infty},f\in\lc$
such that $f_{k}\to f$. Then $(\mu_{f_{k}},\nu_{f_{k}})\to(\mu_{f},\nu_{f})$
cosmically. 
\end{thm}

Finally, once we have Theorems \ref{thm:main-Minkowksi} and \ref{thm:measures-cont},
it is very natural to ask if the converse is also true, and if the
solution $f$ to the functional Minkowski problem is continuous in
the data $(\mu,\nu)$. This turns out to be true, and is our final
main result: 
\begin{thm}
\label{thm:functions-cont}Fix $\left\{ f_{k}\right\} _{k=1}^{\infty},f\in\lc$
and assume that $(\mu_{f_{k}},\nu_{f_{k}})\to(\mu_{f},\nu_{f})$ cosmically.
Then there exists a sequence of translations $\widetilde{f}_{k}(x)=f_{k}(x+v_{k})$
such that $\widetilde{f}_{k}\to f$. 
\end{thm}

In the very special case where $\nu_{f_{k}}\equiv0$ for all $k$
and all the measures $\mu_{f_{k}}$ are supported on a compact set
(the same compact set for all $k$), Theorem \ref{thm:functions-cont}
was previously proved by Klartag (\cite{KlartagLogarithmicallyConcave2014}). 

\subsection{Proof ideas and the structure of this paper}

Very often, Minkowski type theorems are proven using a variational
argument: One defines a functional $F$, and proves that $F$ attains
a minimum and that this minimum is the required solution to the Minkowski
problem. 

In our case, we define $F:\cvx\to(-\infty,\infty]$ by 
\[
F(\phi)=\int_{\RR^{n}}\phi^{\ast}\dd\mu+\int_{\SS^{n-1}}\overline{\phi^{\ast}}\dd\nu-\mu(\RR^{n})\log\int_{\RR^{n}}e^{-\phi}.
\]
 This is inspired by both Alexandrov's proof of the classical Minkowski
Theorem \ref{thm:minkowski}, and by the result of \cite{Cordero-ErausquinKlartagMoment2015}
that used the same functional in the case $\nu=0$ to prove the functional
Minkowski theorem in this case. 

In order to prove that $F$ attains a minimum we equip $\cvx$ with
the topology of epi-convergence, and prove that every minimizing sequence
for $F$ must be uniformly coercive. From there we use known compactness
results to prove that our minimizing sequence has a convergent sub-sequence,
and therefore that $F$ attains a minimum. We believe that thanks
to the use of epi-convergence our argument here is simpler and more
transparent than the corresponding argument of \cite{Cordero-ErausquinKlartagMoment2015},
even though we are dealing with general measures $\mu$ and $\nu$. 

Next, one needs to prove that the minimizer of $F$ solves the Minkowski
problem. This essentially requires computing the derivative $\left.\frac{\dd}{\dd t}\right|_{t=0}F\left((\phi^{\ast}+t\xi)^{\ast}\right)$
for a large enough class of functions $\xi:\RR^{n}\to\RR$. In the
case of convex bodies the computation of the corresponding derivative
relies on the so-called Alexandrov lemma -- see e.g. \cite[Lemma 7.5.3]{SchneiderConvex2013}.
We therefore need a functional version of this lemma, which is far
from trivial. In \cite{UlivelliFirst2024} Ulivelli proved a functional
Alexandrov lemma, using a weighted version of Alexandrov's lemma from
\cite{KryvonosLangharstWeighted2023}. Unfortunately, Ulivelli's theorem
has some assumptions that make it unsuitable for our goal, most crucially
that the convex functions involved have compact domain. Luckily, we
do not need to compute the derivative for all $\xi$ , just for a
large enough family of functions. We do so in Section \ref{sec:alexandrov},
surprisingly using a theorem of Matheron about the Minkowski difference
of convex bodies. Then in Section \ref{sec:Minkowski} we carry out
the proof that $F$ attains a minimizer and use the results of Section
\ref{sec:alexandrov} to conclude that this minimizer is the sought
after solution to the Minkowski theorem.

Next, in Section \ref{sec:cosmic}, we study cosmically continuous
functions and cosmic convergence. We explain where these terms come
from and how to view the pair $(\mu_{f},\nu_{f})$ as a single measure
on the cosmic closure of $\RR^{n}$. This is crucial for the proof
of Theorem \ref{thm:measures-cont}, but we believe it is also of
independent interest. For example, it allows us to rewrite the first
variation formula (\ref{eq:representation}) as 
\begin{equation}
\delta(e^{-\phi},g)=\int_{\partial\epi(\phi)}\widehat{h_{g}}\left(n_{\epi(\phi)}(x,t)\right)e^{-t}\dd\HH^{n}(x,t),\label{eq:variation-intro}
\end{equation}
 where $\widehat{h_{g}}$ is essentially the function $h_{g}$ extended
to the cosmic closure of $\RR^{n}$. Note that in this formula we
no longer have any explicit ``boundary term''. Similar ideas appeared
in \cite{UlivelliFirst2024}, but as far as we can tell our definitions
and the formula above are new.

In Section \ref{sec:continuity} we prove Theorem \ref{thm:measures-cont}.
To explain the main idea behind the proof it is useful to sketch a
proof of the weak continuity of the surface area measure for convex
bodies. This specific result can also be obtained using mixed volumes,
but slightly modifying the problem makes this approach impossible,
so various variants of the following argument appeared in the literature
(See e.g. \cite[Theorem 3.4]{HuangEtAlMinkowski2021}, \cite[Lemma 3.12]{UlivelliFirst2024},
and \cite[Lemma 6]{SchneiderWeighted2024}). Assume $K_{i}\to K$.
We want to prove that $S_{K_{i}}\to S_{K}$, or that for all continuous
functions $\rho:\SS^{n-1}\to\RR$ we have 
\[
\int_{\partial K_{i}}\rho\left(n_{K_{i}}(x)\right)\dd\HH^{n-1}(x)=\int_{\SS^{n-1}}\rho\dd S_{K_{i}}\to\int_{\SS^{n-1}}\rho\dd S_{K}=\int_{\partial K}\rho\left(n_{K}(x)\right)\dd\HH^{n-1}(x).
\]
 To do so we transform the domain of integration from $\partial K$
(or $\partial K_{i}$) to $\SS^{n-1}$ using the change of variables
$\theta\mapsto r_{K}(\theta)\theta$, where $r_{K}(\theta)=\max\left\{ \lambda>0:\ \lambda\theta\in K\right\} $
is the radial function of $K$. Once all the integrals involved are
on the same domain $\SS^{n-1}$ , it is enough to prove that the integrands
converge almost everywhere and use the dominated convergence theorem. 

We want to argue in a similar way using integrals of the form (\ref{eq:variation-intro}),
so we need a cleverly chosen change of variables which transforms
the domain of integration from $\partial\epi(\phi)$ to $\RR^{n}$.
Known constructions do not seem to be appropriate here, so we define
a new kind of a radial function which we call the \emph{curvilinear
radial function} of $\phi$, and show that it can be used to parametrize
$\partial\epi(\phi)$. Constructing this new radial function and proving
its basic properties take up the majority of Section \ref{sec:continuity},
and once this is done Theorem \ref{thm:measures-cont} follows using
the scheme mentioned above. We believe our new radial function can
have more applications, for example in the proof of a more general
Alexandrov-type lemma. 

Finally, in Section \ref{sec:reverse-continuity} we prove Theorem
\ref{thm:functions-cont}. The argument uses again a compactness result
for $\lc$, Theorem \ref{thm:measures-cont}, and the uniqueness part
of Theorem \ref{thm:main-Minkowksi}. This is not very different than
known arguments for theorems like Theorem \ref{thm:continuity-bodies},
but working with log-concave functions does add some difficulties.
In particular, we note as a curious fact that the proof relies on
an isoperimetric inequality for log-concave functions, proved in \cite{MilmanRotemMixed2013}
-- see Proposition \ref{prop:isoperimetric}. Another necessary ingredient
is the computation of $\delta(f,f)$ for $f\in\lc$ -- for convex
bodies this is trivial since volume is $n$-homogeneous, but it is
less obvious for log-concave functions, and indeed $\delta(f,f)$
is not proportional to $\int f$. Luckily the computation of $\delta(f,f)$
was already carried out in \cite{ColesantiFragalaFirst2013} -- see
Proposition \ref{prop:d(f,f)}. 

\subsection{Acknowledgements}

The authors were supported by ISF grant number 2574/24 and NSF-BSF
grant number 2022707.

\section{\label{sec:alexandrov}A partial Alexandrov lemma}

In our solution of the Minkowski problem we will need to compute derivatives
of the form 
\begin{equation}
\left.\frac{\dd}{\dd t}\right|_{t=0}\int_{\RR^{n}}e^{-\left(\phi^{\ast}+t\xi\right)^{\ast}},\label{eq:two-sided-der}
\end{equation}
 where $\phi\in\cvx$ and $\xi:\RR^{n}\to\RR$ is a ``nice enough''
function. The function $\phi^{\ast}+t\xi$ is not necessarily convex,
but we still compute its Legendre transform using the standard formula,
i.e. 
\[
\left(\phi^{\ast}+t\xi\right)^{\ast}(y)=\sup_{x\in\RR^{n}}\left(\left\langle x,y\right\rangle -\left(\phi^{\ast}(x)+t\xi(x)\right)\right).
\]

Now that if $\xi\in\cvx$ and if we define $f=e^{-\phi}$ and $g=e^{-\xi^{\ast}}$
, then $e^{-\left(\phi^{\ast}+t\xi\right)^{\ast}}=f\star\left(t\cdot g\right)$
for all $t>0$. Therefore in this case the derivative we are trying
to compute is exactly the one from the first variation formula (\ref{eq:representation}).
However, formula (\ref{eq:representation}) concerns only positive
values of $t$ (i.e. the right derivative), and for us it would be
crucial to have a two-sided derivative. 

In the case that $\phi^{\ast}$ is the support function of a convex
body and $\xi$ is continuous and $1$-homogeneous, the computation
of (\ref{eq:two-sided-der}) is exactly the well-known Alexandrov
lemma, which is a standard ingredient in the proof of the classical
Minkowski problem. Our goal is therefore to prove a functional extension
of this lemma. As was already mentioned in the introduction such a
functional Alexandrov lemma was recently proved in \cite{UlivelliFirst2024},
but its assumptions are not suitable for our goals. 

Computing the derivative (\ref{eq:two-sided-der}) for\emph{ all}
continuous functions $\xi:\RR^{n}\to\RR$ is a non-trivial task, which
we do not need for our goal so we will not carry out here. Instead
we will only discuss two partial cases. The first case is relatively
straightforward and was already computed in \cite[Proposition 5.4]{RotemAnisotropic2023}:
\begin{prop}
\label{prop:alexandrov-bounded}Fix $\phi\in\cvx$ and define $f=e^{-\phi}$.
Assume that $\xi:\RR^{n}\to\RR$ is \textbf{bounded} and continuous.
Then 
\[
\left.\frac{\dd}{\dd t}\right|_{t=0}\int_{\RR^{n}}e^{-\left(\phi^{\ast}+t\xi\right)^{\ast}}=\int_{\RR^{n}}\xi\dd\mu_{f}.
\]
\end{prop}

The second case we will need is the case where $\xi=h_{L}$ for a
convex body $L\subseteq\RR^{n}$. For this case we need to recall
the definition of the \emph{Minkowski difference}: For convex bodies
$A,B\subseteq\RR^{n}$ their Minkowski difference is defined by 
\[
A\ominus B=\left\{ x\in\RR^{n}:\ x+B\subseteq A\right\} =\bigcap_{x\in B}\left(A-x\right).
\]
 In \cite{MatheronFormule1978}, Matheron proved the following facts
about the Minkowski difference:
\begin{lem}
\label{lem:mathron-difference}Fix convex bodies $A,B\subseteq\RR^{n}$
and consider the function 
\[
\beta(t)=\begin{cases}
\left|A+tB\right| & t\ge0\\
\left|A\ominus\left(\left|t\right|B\right)\right| & t<0,
\end{cases}
\]
(defined on a ray ($-t_{0},\infty)$ such that $\beta(t)>0$ on $(-t_{0},\infty)$
). Then: 
\begin{enumerate}
\item \label{enu:matheron-differentiable}$\beta$ is (two-sided) differentiable
at $t=0$. 
\item \label{enu:matheron-bound}$\beta$ is convex, so in particular $\left|A+tB\right|-\left|A\right|\ge\left|A\right|-\left|A\ominus tB\right|$
for all small enough $t>0$. 
\end{enumerate}
\end{lem}

Using this lemma we prove:
\begin{thm}
\label{thm:alexandrov-support}Fix $\phi\in\cvx$ and define $f=e^{-\phi}$.
Then for all convex bodies $L\subseteq\RR^{n}$ we have 
\begin{equation}
\left.\frac{\dd}{\dd t}\right|_{t=0}\int_{\RR^{n}}e^{-\left(\phi^{\ast}+th_{L}\right)^{\ast}}=\int_{\RR^{n}}h_{L}\dd\mu_{f}+\int_{\SS^{n-1}}h_{L}\dd\nu_{f}.\label{eq:alexandrov-support}
\end{equation}
 
\end{thm}

\begin{proof}
Recall that $e^{-\left(\phi^{\ast}+th_{L}\right)^{\ast}}=f\star\left(t\cdot\oo_{L}\right)$,
so the first variation formula (\ref{eq:representation}) implies
that
\[
\left.\frac{\dd}{\dd t}\right|_{t=0^{+}}\int_{\RR^{n}}e^{-\left(\phi^{\ast}+th_{L}\right)^{\ast}}=\int_{\RR^{n}}h_{L}\dd\mu_{f}+\int_{\SS^{n-1}}h_{L}\dd\nu_{f}.
\]
It is therefore enough to prove that the two-sided derivative exists,
and we can save ourselves a bit of work by not re-computing this derivative. 

Write $f_{t}=e^{-\left(\phi^{\ast}+th_{L}\right)^{\ast}}=f\star\oo_{tL}$.
For $t>0$ we have
\[
f_{t}(x)=\sup_{y\in\RR^{n}}\left(f(x-y)\oo_{tL}(y)\right)=\max_{y\in tL}f(x-y).
\]

If on the other hand $t=-s<0$ then we have
\begin{align*}
\left(\phi^{\ast}+th_{L}\right)^{\ast}(x) & =\left(\phi^{\ast}-h_{sL}\right)^{\ast}(x)=\sup_{y\in\RR^{n}}\left(\left\langle x,y\right\rangle -\phi^{\ast}(y)+h_{sL}(y)\right)\\
 & =\sup_{y\in\RR^{n}}\sup_{z\in sL}\left(\left\langle x,y\right\rangle -\phi^{\ast}(y)+\left\langle z,y\right\rangle \right)\\
 & =\sup_{z\in sL}\sup_{y\in\RR^{n}}\left(\left\langle x+z,y\right\rangle -\phi^{\ast}(y)\right)\\
 & =\sup_{z\in sL}\phi^{\ast\ast}\left(x+z\right)=\sup_{z\in sL}\phi(x+z),
\end{align*}
 and so $f_{t}(x)=\inf_{z\in sL}f\left(x+z\right)$. 

Let us rewrite these identities in terms of level sets. For $u>0$
we define 
\[
K_{u}=\left\{ x\in\RR^{n}:\ f(x)\ge u\right\} ,
\]
 a set we also denote by $\left[f\ge u\right]$. Then for $t>0$ we
have 
\begin{align*}
\left[f_{t}\ge u\right] & =\left\{ x\in\RR^{n}:\ \exists y\in tL,\ f(x-y)\ge u\right\} \\
 & =\left\{ x\in\RR^{n}:\ \exists y\in tL,\ x-y\in K_{u}\right\} =K_{u}+tL,
\end{align*}
 while for $t=-s<0$ we have
\begin{align*}
\left[f_{t}\ge u\right] & =\left\{ x\in\RR^{n}:\ \forall y\in sL,\ f(x+y)\ge u\right\} \\
 & =\left\{ x\in\RR^{n}:\ \forall y\in sL,\ x+y\in K_{u}\right\} =K_{u}\ominus sL.
\end{align*}

We therefore consider the function 
\[
\beta_{u}(t)=\begin{cases}
\left|K_{u}+tL\right| & t\ge0\\
\left|K_{u}\ominus sL\right| & t=-s<0.
\end{cases}
\]
Write $M=\max f$, and note that by layer cake decomposition we have
\[
\lim_{t\to0}\frac{\int f_{t}-\int f}{t}=\lim_{t\to0}\frac{\int_{0}^{M}\left|\left[f_{t}\ge u\right]\right|\dd u-\int_{0}^{M}\left|\left[f\ge u\right]\right|\dd u}{t}=\lim_{t\to0}\int_{0}^{M}\frac{\beta_{u}(t)-\beta_{u}(0)}{t}\dd u,
\]
By Lemma \ref{lem:mathron-difference}(\ref{enu:matheron-differentiable})
we know that 
\[
\lim_{t\to0}\frac{\beta_{u}(t)-\beta_{u}(0)}{t}
\]
 exists. By the second part of the lemma we know that for all $\left|t\right|<t_{0}$
we have 
\[
\left|\frac{\beta_{u}(t)-\beta_{u}(0)}{t}\right|\le\frac{\beta_{u}(\left|t\right|)-\beta_{u}(0)}{\left|t\right|}\le\frac{\beta_{u}(t_{0})-\beta_{u}(0)}{t_{0}}<\infty,
\]
 so we may apply the bounded convergence theorem and exchange the
limit and the integral (For the second inequality we used the known
fact that for $t>0$ the function $t\mapsto\frac{\left|K_{u}+tL\right|-\left|K_{u}\right|}{t}$
is increasing, which can be proven by writing $\left|K_{u}+tL\right|$
as its Steiner polynomial -- See e.g. Section 4.1 of \cite{SchneiderConvex2013}).
We conclude that 
\[
\lim_{t\to0}\frac{\int f_{t}-\int f}{t}=\int_{0}^{M}\beta_{u}^{\prime}(0)\dd u,
\]
so in particular the two-sided derivative exists and the proof is
complete. 
\end{proof}

\section{\label{sec:Minkowski}The Minkowski problem}

We now begin our proof of Theorem \ref{thm:main-Minkowksi}. We begin
by showing that the conditions on the measures $\mu$ and $\nu$ are
indeed necessary: 
\begin{prop}
\label{prop:minkowksi-necessary}For every $f\in\lc$ we have:
\end{prop}

\begin{enumerate}
\item \label{enu:neccessary-non-trivial}$\mu_{f}$ is not identically $0$.
\item \label{enu:neccessary-centered}$\mu_{f}$ has finite first moment
and for all $\theta\in\SS^{n-1}$ we have 
\[
\int_{\RR^{n}}\left\langle x,\theta\right\rangle \dd\mu_{f}(x)+\int_{\SS^{n-1}}\left\langle x,\theta\right\rangle \dd\nu_{f}(x)=0.
\]
\item \label{enu:neccessary-support}$\mu_{f}$ and $\nu_{f}$ are not supported
on a common hyperplane $H\subseteq\RR^{n}$. 
\end{enumerate}
\begin{proof}
Property (\ref{enu:neccessary-non-trivial}) is obvious since $\mu_{f}(\RR^{n})=\int f>0$.
The fact that $\mu_{f}$ has a finite first moment (and that $\mu_{f}$
and $\nu_{f}$ are finite measures) was shown in \cite[Proposition 1.6]{RotemAnisotropic2023}. 

We will show the remaining part of property (\ref{enu:neccessary-centered}),
as well as property (\ref{enu:neccessary-support}), by choosing an
appropriate function $g$ in the variation formula (\ref{eq:representation}).
It will be convenient to use \cite[Proposition 2.4]{RotemAnisotropic2023}:
If $g=\oo_{L}$ for a compact convex set $L\subseteq\RR^{n}$ then
\[
\delta(f,g)=n\int_{0}^{\infty}V_{1}(F_{s},L)\dd s,
\]
 where $F_{s}=\left\{ x\in\RR^{n}:\ f(x)\ge s\right\} $ and $V_{1}$
denotes the first mixed volume, i.e. 
\[
n\cdot V_{1}(K,L)=\lim_{t\to0^{+}}\frac{\left|K+tL\right|-\left|K\right|}{t}.
\]

Choosing $L=\left\{ \theta\right\} $ and $g=\oo_{L}$ we clearly
have $h_{g}(x)=\left\langle x,\theta\right\rangle $ and $V_{1}(K,L)=0$
for all $K$. Therefore 
\[
0=\delta(f,g)=\int_{\RR^{n}}\left\langle x,\theta\right\rangle \dd\mu_{f}(x)+\int_{\SS^{n-1}}\left\langle x,\theta\right\rangle \dd\nu_{f}(x),
\]
 which proves (\ref{enu:neccessary-centered}). 

To show (\ref{enu:neccessary-support}), assume by contradiction that
$\mu_{f}$ and $\nu_{f}$ are supported on $H=\left\{ x:\ \left\langle x,\theta\right\rangle =0\right\} $.
Choose $L=[-\theta,\theta]$ and $g=h_{L}$. Since $h_{g}(x)=h_{L}(x)=\left|\left\langle x,\theta\right\rangle \right|$
it follows from our assumption and the variation formula (\ref{eq:representation})
that 
\[
n\int_{0}^{\infty}V_{1}(F_{s},[-\theta,\theta])\dd s=\delta(f,g)=\int_{\RR^{n}}\left|\left\langle x,\theta\right\rangle \right|\dd\mu_{f}(x)+\int_{\SS^{n-1}}\left|\left\langle x,\theta\right\rangle \right|\dd\nu_{f}(x)=0.
\]
 Therefore we must have $V_{1}(F_{s},[-\theta,\theta])=0$ for all
$s>0$. However it is known that $V_{1}(K,[-\theta,\theta])=\frac{2}{n}\left|\proj_{\theta^{\perp}}K\right|$,
where $\proj$ denotes the orthogonal projection -- this is for example
a special case of \cite[Theorem 5.3.1]{SchneiderConvex2013}. Therefore
$\left|\proj_{\theta^{\perp}}F_{s}\right|=0$ for all $s>0$, which
implies that $F_{s}$ has an empty interior so $\left|F_{s}\right|=0$
as well. But then $\int f=\int_{0}^{\infty}\left|F_{s}\right|\dd s=0$,
which is the required contradiction. 
\end{proof}
We now begin our proof that these conditions are also sufficient for
the existence of a solution to the Minkowski problem. Towards this
goal we fix $\mu$ and $\nu$ that satisfy the conditions of Theorem
\ref{thm:main-Minkowksi} and define a functional $F:\cvx\to(-\infty,\infty]$
by 
\begin{equation}
F(\phi)=\int_{\RR^{n}}\phi^{\ast}\dd\mu+\int_{\SS^{n-1}}\overline{\phi^{\ast}}\dd\nu-\mu(\RR^{n})\log\int_{\RR^{n}}e^{-\phi}.\label{eq:main-functional}
\end{equation}
 We will show that $F$ attains a minimum at some function $\phi_{0}\in\cvx$
and that our sought after solution is $f=ce^{-\phi_{0}}$ for some
$c>0$. In order to prove that $F$ attains a minimum it will be useful
to equip the space $\cvx$ with a topology:
\begin{defn}
\label{def:epi-convergence}Fix $\left\{ \phi_{k}\right\} _{k=1}^{\infty},\phi\in\cvx$.
We say that $\phi_{k}$ epi-converges to $\phi$ as $k\to\infty$
if:
\begin{enumerate}
\item For all $x\in\RR^{n}$ and all sequences $x_{k}\to x$ we have $\liminf_{k\to\infty}\phi_{k}(x_{k})\ge\phi(x)$. 
\item For all $x\in\RR^{n}$ there exists a sequence $x_{k}\to x$ such
that $\limsup_{k\to\infty}\phi_{k}(x_{k})\le\phi(x)$. 
\end{enumerate}
In this case we simply write $\phi_{k}\xrightarrow{k\to\infty}\phi$
or even $\phi_{k}\to\phi$. 
\end{defn}

The notion of epi-convergence is well-known in fields like convex
analysis and optimization, and goes back to the work of Wijsman (\cite{WijsmanConvergence1964,WijsmanConvergence1966})
-- see e.g. \cite{RockafellarWetsVariational1998} for many results
on epi-convergence and historical remarks. In particular, to explain
the name epi-convergence, recall that the epigraph of $\phi\in\cvx$
is given by 
\begin{equation}
\epi(\phi)=\left\{ (x,t)\in\RR^{n}\times\RR:\ \phi(x)\le t\right\} \subseteq\RR^{n+1}.\label{eq:epi-def}
\end{equation}
 Then $\phi_{k}\to\phi$ if and only if $\epi(\phi_{k})\to\epi(\phi)$
in the sense of Painlevé--Kuratowski (\cite[Proposition 7.2]{RockafellarWetsVariational1998}).
For functions in $\cvx$ this just means that $\epi(\phi_{k})\cap\overline{B}_{R}(0)\to\epi(\phi)\cap\overline{B}_{R}(0)$
in the Hausdorff metric for all large enough $R>0$ (\cite[Exercise 4.16]{RockafellarWetsVariational1998}). 

In the field of convex geometry the notion of epi-convergence was
used by Colesanti, Ludwig and Mussnig in \cite{ColesantiEtAlValuations2019}
and subsequent papers (\cite{ColesantiEtAlHadwiger2024}, \cite{ColesantiEtAlHadwiger2023},
\cite{ColesantiEtAlHadwiger2022}, \cite{ColesantiEtAlHadwiger2023a}).
These works present compelling evidence that epi-convergence should
be used as the functional extension of the Hausdorff topology. Our
work can be viewed as further evidence in this direction. Indeed,
as mentioned above in the case $\nu\equiv0$ our functional $F$ coincides
with the one used by Cordero-Erausquin and Klartag (\cite{Cordero-ErausquinKlartagMoment2015}).
However, they only worked with pointwise convergence, and we believe
that our use epi-convergence our proof leads to significant simplifications
of the proof even in this special case. 

One advantage of using epi-convergence is that there are known compactness
theorems for this topology, serving as functional analogues of the
Blaschke selection theorem. We will use the following formulation,
which is due to Li and Mussnig (\cite{LiMussnigMetrics2022}):
\begin{thm}[{\cite[Theorem 2.15]{LiMussnigMetrics2022}}]
\label{thm:selection}Fix a sequence $\left\{ \phi_{k}\right\} _{k=1}^{\infty}\subseteq\cvx$
such that:
\begin{enumerate}
\item $\sup_{k}\left(\min\phi_{k}\right)<\infty$. 
\item The sequence is uniformly coercive: There exists $a>0$ and $b\in\RR^{n}$
such that $\phi_{k}(x)\ge a\left|x\right|+b$ for all $k$ and all
$x\in\RR^{n}$. 
\end{enumerate}
Then there exists a sub-sequence $\left\{ \phi_{k_{\ell}}\right\} _{\ell=1}^{\infty}$
such that $\phi_{k_{\ell}}\to\phi$ for a coercive, lower semicontinuous
convex function $\phi:\RR^{n}\to(-\infty,\infty]$. 
\end{thm}

As a technical point, it may happen that $\int e^{-\phi}=0$, so $\phi$
does not have to belong to $\cvx$. In our applications of Theorem
\ref{thm:selection} it will be easy to rule this option out. 

We can now state the main technical lemma required to prove that $F$
admits a minimizer:
\begin{lem}
\label{lem:functional-bound}Assume $\mu$ and $\nu$ satisfy the
assumptions of Theorem \ref{thm:main-Minkowksi}. Then:
\begin{enumerate}
\item \label{enu:norm-comparison-one}There exists a constant $c>0$ depending
on $\mu$ and $\nu$ such that 
\begin{equation}
\int_{\RR^{n}}\left|\left\langle x,y\right\rangle \right|\dd\mu(x)+\int_{\SS^{n-1}}\left|\left\langle x,y\right\rangle \right|\dd\nu(x)\ge c\left|y\right|\label{eq:norm-compare}
\end{equation}
 for all $y\in\RR^{n}$. 
\item \label{enu:functional-bound}Assume $c>0$ satisfies (\ref{eq:norm-compare}).
Then for every $\phi\in\cvx$ with $\min\phi=\phi(0)$ and for all
$x\in\RR^{n}$ we have 
\begin{equation}
\phi(x)\ge\frac{1}{\mu(\RR^{n})}\left(\frac{c}{2}\left|x\right|-\int_{\RR^{n}}\phi^{\ast}\dd\mu-\int_{\SS^{n-1}}\overline{\phi^{\ast}}\dd\nu\right).\label{eq:unif-coercive}
\end{equation}
 
\end{enumerate}
\end{lem}

\begin{proof}
For the first part, it is enough to observe that the function 
\[
p(y)=\int_{\RR^{n}}\left|\left\langle x,y\right\rangle \right|\dd\mu(x)+\int_{\SS^{n-1}}\left|\left\langle x,y\right\rangle \right|\dd\nu(x)
\]
 is a norm on $\RR^{n}$. Indeed, it is clearly a semi-norm as the
sum of semi-norms, and if $p(y)=0$ for $y\ne0$ then $\mu$ and $\nu$
are supported on $y^{\perp}$ which contradicts our assumption. Since
all norms on $\RR^{n}$ are equivalent it follows that $p(y)\ge c\left|y\right|$
for some $c>0$. 

For the second assertion, set 
\[
G(\phi)=\int_{\RR^{n}}\phi^{\ast}\dd\mu+\int_{\SS^{n-1}}\overline{\phi^{\ast}}\dd\nu.
\]
For all $\lambda\in\RR$ we have $\left(\phi+\lambda\right)^{\ast}=\phi^{\ast}-\lambda$,
and therefore $G(\phi+\lambda)=G(\phi)-\lambda\mu(\RR^{n})$. It follows
that the validity of (\ref{eq:unif-coercive}) doesn't change when
$\phi$ is replaced by $\phi+\lambda$, so we may assume without loss
of generality that $\min\phi=\phi(0)=0$. 

Fix $x_{0}\in\RR^{n}$ with $\phi(x_{0})<\infty$, and consider the
function 
\[
\rho(x)=\begin{cases}
t\cdot\phi(x_{0}) & x=tx_{0},\ 0\le t\le1\\
\infty & \text{otherwise.}
\end{cases}
\]
 Since $\phi$ is convex and $\phi(0)=0$ we clearly have $\phi\le\rho$,
and therefore $\phi^{\ast}\ge\rho^{\ast}$. A direct computation gives
\begin{align*}
\rho^{\ast}(x) & =\sup_{0\le t\le1}\left(\left\langle x,tx_{0}\right\rangle -\rho(tx_{0})\right)\\
 & =\sup_{0\le t\le1}\left[t\left(\left\langle x,x_{0}\right\rangle -\phi(x_{0})\right)\right]=\left[\left\langle x,x_{0}\right\rangle -\phi(x_{0})\right]_{+},
\end{align*}
 where $a_{+}=\max(a,0)$. Therefore $\overline{\rho^{\ast}}(x)=\left\langle x,x_{0}\right\rangle _{+}$.
We can now bound $G(\phi)$ from below as 
\begin{align*}
G(\phi) & \ge G(\rho)=\int_{\RR^{n}}\left[\left\langle x,x_{0}\right\rangle -\phi(x_{0})\right]_{+}\dd\mu(x)+\int_{\SS^{n-1}}\left\langle x,x_{0}\right\rangle _{+}\dd\nu(x)\\
 & \ge\int_{\RR^{n}}\left\langle x,x_{0}\right\rangle _{+}\dd\mu(x)+\int_{\SS^{n-1}}\left\langle x,x_{0}\right\rangle _{+}\dd\nu(x)-\phi(x_{0})\cdot\mu(\RR^{n}),
\end{align*}
 where in the last inequality we used the fact that $(a-b)_{+}\ge a_{+}-b$
whenever $b\ge0$. Using the fact that $a_{+}=\frac{a+\left|a\right|}{2}$
and that $\mu+\nu$ is centered we obtain 
\begin{align*}
G(\phi) & \ge\int_{\RR^{n}}\frac{\left\langle x,x_{0}\right\rangle +\left|\left\langle x,x_{0}\right\rangle \right|}{2}\dd\mu(x)+\int_{\SS^{n-1}}\frac{\left\langle x,x_{0}\right\rangle +\left|\left\langle x,x_{0}\right\rangle \right|}{2}\dd\nu(x)-\phi(x_{0})\cdot\mu(\RR^{n})\\
 & =\frac{1}{2}\int_{\RR^{n}}\left|\left\langle x,x_{0}\right\rangle \right|\dd\mu(x)+\frac{1}{2}\int_{\SS^{n-1}}\left|\left\langle x,x_{0}\right\rangle \right|\dd\nu(x)-\phi(x_{0})\mu(\RR^{n})\ge\frac{c}{2}\left|x_{0}\right|-\phi(x_{0})\mu(\RR^{n}),
\end{align*}
 and rearranging we obtain $\phi(x_{0})\ge\frac{1}{\mu(\RR^{n})}\left(\frac{c}{2}\left|x_{0}\right|-G(\phi)\right)$
as claimed. 
\end{proof}
We will also need the following simple invariance property:
\begin{lem}
\label{lem:functional-invariance}Assume $\mu+\nu$ is centered, and
let $F$ be the functional defined by (\ref{eq:main-functional}).
Fix $\phi\in\cvx$, $v\in\RR^{n}$ and $b\in\RR$ and define $\psi(x)=\phi(x+v)+b$.
Then $F\left(\psi\right)=F\left(\phi\right)$. 
\end{lem}

\begin{proof}
We have 
\begin{align*}
\psi^{\ast}(x) & =\sup_{y\in\RR^{n}}\left[\left\langle x,y\right\rangle -\phi(y+v)-b\right]=\sup_{z\in\RR^{n}}\left[\left\langle x,z-v\right\rangle -\phi(z)-b\right]\\
 & =\phi^{\ast}\left(x\right)-\left\langle x,v\right\rangle -b.
\end{align*}
 Therefore $\overline{\psi^{\ast}}(x)=\overline{\phi^{\ast}}(x)-\left\langle x,v\right\rangle $,
and we may compute 
\begin{align*}
F(\psi) & =\int_{\RR^{n}}\left(\phi^{\ast}-\left\langle x,v\right\rangle -b\right)\dd\mu+\int_{\SS^{n-1}}\left(\overline{\phi^{\ast}}-\left\langle x,v\right\rangle \right)\dd\nu-\mu(\RR^{n})\log\int_{\RR^{n}}e^{-\phi-b}\\
 & \stackrel{\left(\ast\right)}{=}\int_{\RR^{n}}\phi^{\ast}\dd\mu-b\mu(\RR^{n})+\int_{\SS^{n-1}}\overline{\phi^{\ast}}\dd\nu-\mu(\RR^{n})\left(\log\left(\int_{\RR^{n}}e^{-\phi}\right)-b\right)=F(\phi)
\end{align*}
 as claimed. Note that in the equality $\left(\ast\right)$ we used
the fact that $\mu+\nu$ is centered. 
\end{proof}
We are ready to prove:
\begin{prop}
\label{prop:minimizer-exists}Assume $\mu$ and $\nu$ satisfy the
assumptions of Theorem \ref{thm:main-Minkowksi}. Then the functional
$F:\cvx\to(-\infty,\infty]$ defined by (\ref{eq:main-functional})
attains a minimum on $\cvx$. 
\end{prop}

\begin{proof}
Choose a minimizing sequence $\left\{ \phi_{k}\right\} _{k=1}^{\infty}$
, i.e. a sequence such that $F(\phi_{k})\to\inf F$. Using the invariance
property of Lemma \ref{lem:functional-invariance} we may assume that
$\int e^{-\phi_{k}}=1$ and that $\min\phi=\phi(0)$. We may then
apply Lemma \ref{lem:functional-bound} and deduce that 
\[
\phi_{k}(x)\ge\frac{1}{\mu(\RR^{n})}\left(c\left|x\right|-F(\phi_{k})\right)
\]
 for a constant $c>0$ independent of $k$. Since $\left\{ \phi_{k}\right\} _{k=1}^{\infty}$
is a minimizing sequence clearly $\left\{ F(\phi_{k})\right\} _{k=1}^{\infty}$
is bounded from above, so it follows that $\left\{ \phi_{k}\right\} _{k=1}^{\infty}$
is uniformly coercive: $\phi_{k}(x)\ge a\left|x\right|+b$ for all
$k$ and all $x\in\RR^{n}$, where the constants $a$ and $b$ do
not depend on $k$. 

Next we claim that $\sup_{k}\left(\min\phi_{k}\right)=\sup_{k}\left(\phi_{k}(0)\right)<\infty$.
Indeed, if this is not the case then there exists a subsequence $\left\{ \phi_{k_{j}}\right\} _{j=1}^{\infty}$
such that $\min\phi_{k_{j}}\to\infty$. But then 
\[
1=\lim_{j\to\infty}\int e^{-\phi_{k_{j}}}\stackrel{\left(\ast\right)}{=}\int e^{-\lim_{j\to\infty}\phi_{k_{j}}}=\int0=0,
\]
 which is a contradiction. The exchange of limit and integral in $\left(\ast\right)$
is justified by the dominated convergence theorem, since all the functions
$e^{-\phi_{k}}$ are bounded by the integrable function $e^{-\left(a\left|x\right|+b\right)}$. 

It now follows from Theorem \ref{thm:selection} that we may pass
to a subsequence and assume without loss of generality that $\phi_{k}\to\phi$
for a coercive, lower semicontinuous convex function $\phi:\RR^{n}\to(-\infty,\infty]$.
By Lemma \cite[Lemma 15]{ColesantiEtAlValuations2019} we have $\int e^{-\phi}=\lim_{k\to\infty}\int e^{-\phi_{k}}=1$,
so $\phi\in\cvx$. It was shown already by Wijsman (\cite[Theorem 6.2]{WijsmanConvergence1966},
see also \cite[Theorem 11.34]{RockafellarWetsVariational1998}) that
we also have $\phi_{k}^{\ast}\to\phi^{\ast}$, and then by e.g. \cite[Theorem 7.53]{RockafellarWetsVariational1998}
we also have $\overline{\phi_{k}^{\ast}}\to\overline{\phi^{\ast}}$
(to make sense of this last convergence either extend the definition
of epi-convergence to functions defined on $\SS^{n-1}$ in the obvious
way, or consider the functions $\overline{\phi_{k}^{\ast}}$ as $1$-homogeneous
functions on $\RR^{n}$). In particular, we have the pointwise estimate
$\phi^{\ast}(x)\le\liminf_{k\to\infty}\phi_{k}^{\ast}(x)$ and similarly
$\overline{\phi^{\ast}}(x)\le\liminf_{k\to\infty}\overline{\phi_{k}^{\ast}}(x)$
.  

We note that the functions $\left\{ \phi_{k}^{\ast}\right\} $ are
uniformly bounded from below, since 
\[
\inf_{k}\left(\inf\phi_{k}^{\ast}\right)=\inf_{k}\left(-\phi_{k}^{\ast\ast}(0)\right)=-\sup_{k}\phi_{k}(0)>-\infty.
\]
It follows that also $\overline{\phi_{k}^{\ast}}\ge0$, and we may
apply Fatou's lemma and deduce that 
\begin{align*}
F(\phi) & \le\int_{\RR^{n}}\left(\liminf_{k\to\infty}\phi_{k}^{\ast}\right)\dd\mu+\int_{\SS^{n-1}}\left(\liminf_{k\to\infty}\overline{\phi_{k}^{\ast}}\right)\dd\nu\\
 & \le\liminf_{k\to\infty}\left(\int_{\RR^{n}}\phi_{k}^{\ast}\dd\mu+\int_{\SS^{n-1}}\overline{\phi_{k}^{\ast}}\dd\nu\right)=\lim_{k\to\infty}F(\phi_{k})=\inf F.
\end{align*}
 Therefore $\phi$ is the required minimizer. 
\end{proof}
Using the existence of minimizer and the results of the previous section
we can now prove Theorem \ref{thm:main-Minkowksi}:
\begin{proof}[Proof of Theorem \ref{thm:main-Minkowksi}]

The fact that the conditions of the theorem are necessary is exactly
Proposition \ref{prop:minkowksi-necessary}, and the fact that the
solution to the Minkowski problem is unique up to translations was
previously proved in \cite[Corollary 3.3]{RotemAnisotropic2023}.
Therefore we only need to prove the existence of a solution under
our assumptions on $\mu$ and $\nu$. 

Let $\phi$ be the minimizer of $F$ over $\cvx$, whose existence
is guaranteed by Proposition \ref{prop:minimizer-exists}. By Lemma
\ref{lem:functional-invariance} we may add a constant to $\phi$
and assume that $\int_{\RR^{n}}e^{-\phi}=\mu(\RR^{n})$. Define $f=e^{-\phi}$.
We will prove that $f$ is the required solution, that is $\mu_{f}=\mu$
and $\nu_{f}=\nu$. 

Fix a continuous function $\xi:\RR^{n}\to\RR$, which is either bounded
or the support function of a convex body. For $t\in\RR$ we define
$\phi_{t}=(\phi^{\ast}+t\xi)^{\ast}\in\cvx$. Since $\phi$ is a minimizer
of $F$ we have $F(\phi_{t})\ge F(\phi)$ for all $t$, so the function
\[
\alpha(t)=F(\phi_{t})=\int_{\RR^{n}}\phi_{t}^{\ast}\dd\mu+\int_{\SS^{n-1}}\overline{\phi_{t}^{\ast}}\dd\nu-\mu(\RR^{n})\log\int_{\RR^{n}}e^{-\phi_{t}}
\]
 attains a minimum at $t=0$. Since $\phi_{t}^{\ast}=(\phi^{\ast}+t\xi)^{\ast\ast}\le\phi^{\ast}+t\xi$
it follows that the function 
\[
\beta(t)=\int_{\RR^{n}}\left(\phi^{\ast}+t\xi\right)\dd\mu+\int_{\SS^{n-1}}\left(\overline{\phi^{\ast}}+t\overline{\xi}\right)\dd\nu-\mu(\RR^{n})\log\int_{\RR^{n}}e^{-\phi_{t}}
\]
 satisfies $\alpha(t)\le\beta(t)$ for all $t$, and $\alpha(0)=\beta(0)$.
Therefore the function $\beta$ also attains a minimum at $t=0$,
and $\beta'(0)=0$ whenever $\beta$ is differentiable at $t=0$. 

Assume $\xi$ is continuous and bounded, so $\overline{\xi}=0$. It
follows from Proposition \ref{prop:alexandrov-bounded} that $\beta$
is differentiable at the origin, and that 
\[
0=\beta'(0)=\int_{\RR^{n}}\xi\dd\mu-\mu(\RR^{n})\cdot\frac{1}{\int_{\RR^{n}}e^{-\phi}}\int_{\RR^{n}}\xi\dd\mu_{f}=\int_{\RR^{n}}\xi\dd\mu-\int_{\RR^{n}}\xi\dd\mu_{f}.
\]
 Since this holds for all bounded and continuous $\xi$, it follows
from the Riesz representation theorem that $\mu_{f}=\mu$. 

Now assume $\xi$ is a support function of a convex body, so $\overline{\xi}=\xi$.
It follows from Theorem \ref{thm:alexandrov-support} that $\beta$
is again differentiable at $t=0$ and 
\[
0=\beta'(0)=\int_{\RR^{n}}\xi\dd\mu+\int_{\SS^{n-1}}\xi\dd\nu-\frac{\mu(\RR^{n})}{\int_{\RR^{n}}e^{-\phi}}\left(\int_{\RR^{n}}\xi\dd\mu_{f}+\int_{\SS^{n-1}}\xi\dd\nu_{f}\right).
\]
 Using the fact that $\mu(\RR^{n})=\int_{\RR^{n}}e^{-\phi}$ and that
$\mu=\mu_{f}$ we deduce that for all support functions $\xi$ we
have 
\begin{equation}
\int_{\SS^{n-1}}\xi\dd\nu=\int_{\SS^{n-1}}\xi\dd\nu_{f}.\label{eq:nu-equality}
\end{equation}
By linearity (\ref{eq:nu-equality}) also holds for all differences
of support functions, a class that includes all $C^{2}$-smooth functions
on $\SS^{n-1}$. As $C^{2}$ functions are dense in the space of continuous
functions (with the supremum norm) it follows that (\ref{eq:nu-equality})
holds for continuous functions $\xi:\SS^{n-1}\to\RR$, and so again
by Riesz we have $\nu=\nu_{f}$. This completes the proof. 
\end{proof}

\section{\label{sec:cosmic}Cosmic convergence}

We now turn our attention to the continuity of functional surface
area measures. To discuss continuity we need appropriate notions of
convergence both on the class $\lc$ and on pairs of measures of the
form $(\mu_{f},\nu_{f})$. The natural identification between $\cvx$
and $\lc$ gives us a notion of convergence on $\lc$:
\begin{defn}
Fix $\left\{ f_{k}\right\} _{k=1}^{\infty},f\in\lc$. We say that
$f_{k}\to f$ if $(-\log f_{k})\to(-\log f)$ in the sense of epi-convergence. 
\end{defn}

Since the map $x\mapsto-\log x$ is decreasing, our convergence on
$\lc$ is not by itself epi-convergence, but the symmetric notion
of hypo-convergence. Nonetheless we use the same simple notation $f_{k}\to f$,
which should not cause any confusion. 

For pairs of measures we use the (new) notion of cosmic convergence
as was defined in Definition \ref{def:cosmic-conv} of the introduction.
Our first goal is to explain this convergence, and in particular to
explain the name ``cosmic convergence''. Consider the embedding
of $\RR^{n}$ into $\RR^{n+1}$ as $\RR^{n}\cong\RR^{n}\times\left\{ -1\right\} \subseteq\RR^{n+1}$.
Also consider the lower half-sphere 
\[
\SS_{-}^{n}=\left\{ y=(y_{1},y_{2},\ldots,y_{n+1})\in\SS^{n}:\ y_{n+1}<0\right\} \subseteq\RR^{n}.
\]
We write a general point in $\RR^{n+1}$ as $(x,t)$ for $x\in\RR^{n}$
and $t\in\RR$. The spaces $\RR^{n}\cong\RR^{n}\times\left\{ -1\right\} $
and $\SS_{-}^{n}$ are then homeomorphic using the gnomonic projection
$\Phi:\SS_{-}^{n}\to\RR^{n}$ defined by $\Phi(x,t)=-\frac{x}{t}=\frac{x}{\left|t\right|}$.
The inverse map $\Phi^{-1}:\RR^{n}\to\SS_{-}^{n}$ is defined by 
\[
\Phi^{-1}(x)=\left(\frac{x}{\sqrt{1+\left|x\right|^{2}}},-\frac{1}{\sqrt{1+\left|x\right|^{2}}}\right).
\]

Now the space $\SS_{-}^{n}$ has a natural compactification, which
is its usual closure $\overline{\SS_{-}^{n}}$ in $\RR^{n+1}$, i.e.
the closed lower half-sphere. Intuitively this is a compactification
of $\RR^{n}$ with a point at infinity ``in every direction''. This
compactification is known as the cosmic closure of $\RR^{n}$, and
its identification with $\overline{\SS_{-}^{n}}$ is sometimes referred
to as the hemispherical model -- see \cite[Chapter 3A]{RockafellarWetsVariational1998}. 

To every function $\xi:\RR^{n}\to\RR$ we associate a function $\widehat{\xi}:\SS_{-}^{n}\to\RR$
in the following way: We first extend $\xi$ from $\RR^{n}\cong\RR^{n}\times\left\{ -1\right\} $
to a $1$-homogeneous function on the lower half-space $\left\{ (x,t)\in\RR^{n+1}:\ t<0\right\} $,
and then restrict this extension to $\SS_{-}^{n}$. Explicitly, we
have 
\begin{equation}
\widehat{\xi}(x,t)=\frac{\xi\left(\Phi(x,t)\right)}{\left|\left(\Phi(x,t),-1\right)\right|}=\left|t\right|\xi\left(\frac{x}{\left|t\right|}\right).\label{eq:nohat-to-hat}
\end{equation}
 This operation is clearly a bijection between functions on $\RR^{n}$
and functions on $\SS_{-}^{n}$, with the inverse operation given
by 
\begin{equation}
\xi(x)=\left|(x,-1)\right|\cdot\widehat{\xi}\left(\Phi^{-1}(x)\right)=\sqrt{\left|x\right|^{2}+1}\cdot\widehat{\xi}\left(\frac{x}{\sqrt{1+\left|x\right|^{2}}},-\frac{1}{\sqrt{1+\left|x\right|^{2}}}\right).\label{eq:hat-to-nohat}
\end{equation}

We can now explain Definitions \ref{def:cosmically-continuous} and
\ref{def:cosmic-conv}. We have:
\begin{prop}
\label{prop:cosmic-characterization}A function $\xi:\RR^{n}\to\RR$
is cosmically continuous in the sense of Definition \ref{def:cosmically-continuous}
if and only if $\widehat{\xi}:\SS_{-}^{n}\to\RR$ can be extended
to a continuous function on the cosmic closure $\overline{\SS_{-}^{n}}$.
Moreover, in this case the extension $\widehat{\xi}:\overline{\SS_{-}^{n}}\to\RR$
satisfies $\widehat{\xi}(\theta,0)=\overline{\xi}(\theta)$ for all
$\theta\in\SS^{n-1}$. 
\end{prop}

In other words, cosmically continuous functions are simply continuous
functions on $\RR^{n}$ that can be extended continuously to its cosmic
closure (under our identification of $\xi$ and $\widehat{\xi}$).
The proof of Proposition \ref{prop:cosmic-characterization} is a
straightforward exercise in topology, which we nonetheless include
here for completeness:
\begin{proof}
Assume first that $\widehat{\xi}$ can be extended to a continuous
function $\overline{\SS_{-}^{n}}$, which we also denote by $\widehat{\xi}$.
Formula (\ref{eq:hat-to-nohat}) immediately shows that $\xi$ is
continuous on $\RR^{n}$. Moreover for all $\theta\in\SS^{n-1}$ we
have 
\begin{align*}
\lim_{t\to\infty}\frac{\xi(t\theta)}{t} & =\lim_{t\to\infty}\frac{\sqrt{t^{2}+1}\widehat{\xi}\left(\frac{t\theta}{\sqrt{t^{2}+1}},-\frac{1}{\sqrt{t^{2}+1}}\right)}{t}\\
 & =\lim_{t\to\infty}\frac{\sqrt{t^{2}+1}}{t}\cdot\lim_{t\to\infty}\widehat{\xi}\left(\frac{t}{\sqrt{t^{2}+1}}\theta,-\frac{1}{\sqrt{t^{2}+1}}\right)=\widehat{\xi}(\theta,0)
\end{align*}
 where we used the continuity of $\widehat{\xi}$. Moreover, since
$\widehat{\xi}$ is continuous on the compact set $\overline{\SS_{-}^{n}}$
it is uniformly continuous. Using this and the fact that 
\[
\lim_{t\to\infty}\sup_{\theta\in\SS^{n-1}}\left|\left(\frac{t}{\sqrt{t^{2}+1}}\theta,-\frac{1}{\sqrt{t^{2}+1}}\right)-(\theta,0)\right|=0
\]
 we see that the $\lim_{t\to\infty}\frac{\xi(t\theta)}{t}$ exists
uniformly in $\theta\in\SS^{n-1}$, and hence $\xi$ is cosmically
continuous in the sense of Definition \ref{def:cosmically-continuous}. 

For the converse, assume $\xi:\RR^{n}\to\RR$ is cosmically continuous.
Our goal is to prove that the function $h:\overline{\SS_{-}^{n}}\to\RR$
defined by 
\[
h(x,t)=\begin{cases}
\widehat{\xi}(x,t) & t<0\\
\overline{\xi}(x) & t=0.
\end{cases}
\]
 is the required continuous extension of $\widehat{\xi}$. It is clearly
continuous on $\SS_{-}^{n}$ by (\ref{eq:nohat-to-hat}), so we only
need to check its continuity at every point of the form $(\theta,0)$. 

We note that $\overline{\xi}$ is continuous on $\SS^{n-1}$ as a
uniform limit of the continuous functions $\theta\mapsto\frac{\xi(\lambda\theta)}{\lambda}$.
It is therefore enough to fix a sequence $\left\{ (x_{k},t_{k})\right\} \subseteq\SS_{-}^{n}$
such that $(x_{k},t_{k})\to(\theta,0)$ and prove that $h(x_{k},t_{k})\to h(\theta,0)$.
Define $\theta_{k}=\frac{x_{k}}{\left|x_{k}\right|}\to\theta$, and
write
\begin{align*}
\left|h(x_{k},t_{k})-h(\theta,0)\right| & =\left|\widehat{\xi}(x_{k},t_{k})-\overline{\xi}(\theta)\right|\le\left|\widehat{\xi}(x_{k},t_{k})-\overline{\xi}\left(\theta_{k}\right)\right|+\left|\overline{\xi}\left(\theta_{k}\right)-\overline{\xi}(\theta)\right|.
\end{align*}
 By the continuity of $\overline{\xi}$ we have $\left|\overline{\xi}\left(\theta_{k}\right)-\overline{\xi}(\theta)\right|\to0$.
For the first term we write 
\begin{align*}
\left|\widehat{\xi}(x_{k},t_{k})-\overline{\xi}(\theta_{k})\right| & =\left|\left|t_{k}\right|\xi\left(\frac{x_{k}}{\left|t_{k}\right|}\right)-\overline{\xi}(\theta_{k})\right|\\
 & =\left|\left|x_{k}\right|\frac{\left|t_{k}\right|}{\left|x_{k}\right|}\xi\left(\frac{\left|x_{k}\right|\theta_{k}}{\left|t_{k}\right|}\right)-\left|x_{k}\right|\overline{\xi}(\theta_{k})+\overline{\xi}(\theta_{k})\left(\left|x_{k}\right|-1\right)\right|\\
 & \le\left|x_{k}\right|\cdot\left|\frac{\xi(s_{k}\theta_{k})}{s_{k}}-\overline{\xi}(\theta_{k})\right|+\left(\max\overline{\xi}\right)\cdot\left(\left|x_{k}\right|-1\right)\\
 & \xrightarrow{k\to\infty}1\cdot0+\left(\max\overline{\xi}\right)(1-1)=0
\end{align*}
 where $s_{k}=\frac{\left|x_{k}\right|}{\left|t_{k}\right|}\to\infty$,
and where we used the uniform convergence of $\frac{\xi(\lambda\theta)}{\lambda}$
to $\overline{\xi}$. This shows that $h(x_{k},t_{k})\to h(\theta,0)$,
finishing the proof. 
\end{proof}
We can now also better understand the notion of cosmic convergence.
To every Borel measure $\mu$ on $\RR^{n}$ we can associate a measure
$\widehat{\mu}$ on $\SS_{-}^{n}$ via the relation 
\begin{equation}
\int_{\RR^{n}}\xi\dd\mu=\int_{\SS^{n-1}}\widehat{\xi}\dd\widehat{\mu}\label{eq:measure-identification}
\end{equation}
 for all cosmically continuous functions $\xi:\RR^{n}\to\RR$. Every
Borel measure $\nu$ on $\SS^{n-1}$ can also be considered as a measure
$\widehat{\nu}$ on the equator $\overline{\SS_{-}^{n}}\setminus\SS_{-}^{n}=\left\{ (x,0):\ x\in\SS^{n-1}\right\} $
by identifying this equator with $\SS^{n-1}$ in the obvious way.
Together $\widehat{\mu}+\widehat{\nu}$ is a single measure on the
cosmic closure $\overline{\SS_{-}^{n}}$ which is in a natural one-to-one
correspondence with the pair $\left(\mu,\nu\right)$. We then have:
\begin{prop}
\label{prop:cosmic-weak}$\left(\mu_{k,}\nu_{k}\right)\to\left(\mu,\nu\right)$
cosmically if and only if $\widehat{\mu_{k}}+\widehat{\nu_{k}}\to\widehat{\mu}+\widehat{\nu}$
weakly on $\overline{\SS_{-}^{n}}$. 
\end{prop}

In other words, the cosmic convergence $\left(\mu_{k,}\nu_{k}\right)\to\left(\mu,\nu\right)$
is the same as weak convergence on the cosmic closure $\overline{\SS_{-}^{n}}$
of $\RR^{n}$, under our identification of the pair $(\mu,\nu)$ with
the measure $\widehat{\mu}+\widehat{\nu}$. 
\begin{proof}
This is an exercise in expanding the definitions. By definition, $\left(\mu_{k,}\nu_{k}\right)\to\left(\mu,\nu\right)$
cosmically if for every cosmically continuous function $\xi:\RR^{n}\to\RR$
we have
\begin{equation}
\int_{\RR^{n}}\xi\dd\mu_{k}+\int_{\SS^{n-1}}\overline{\xi}\dd\nu_{k}\xrightarrow{k\to\infty}\int_{\RR^{n}}\xi\dd\mu+\int_{\SS^{n-1}}\overline{\xi}\dd\nu.\label{eq:cosmic-equiv-def}
\end{equation}
By (\ref{eq:measure-identification}) and the ``moreover'' part
of Proposition \ref{prop:cosmic-characterization} we have 
\begin{align*}
\int_{\RR^{n}}\xi\dd\mu+\int_{\SS^{n-1}}\overline{\xi}\dd\nu & =\int_{\SS_{-}^{n}}\widehat{\xi}\widehat{\mu}+\int_{\SS^{n-1}}\widehat{\xi}(\theta,0)\dd\nu(\theta)\\
 & =\int_{\SS_{-}^{n}}\widehat{\xi}\dd\widehat{\mu}+\int_{\overline{\SS_{-}^{n}}\setminus\SS_{-}^{n}}\widehat{\xi}\dd\widehat{\nu}=\int_{\overline{\SS_{-}^{n}}}\widehat{\xi}\dd\left(\widehat{\mu}+\widehat{\nu}\right).
\end{align*}
 The same of course is true for $\left(\mu_{k},\nu_{k}\right)$, so
the cosmic convergence (\ref{eq:cosmic-equiv-def}) is equivalent
to 
\[
\int_{\overline{\SS_{-}^{n}}}\widehat{\xi}\dd\left(\widehat{\mu_{k}}+\widehat{\nu_{k}}\right)\to\int_{\overline{\SS_{-}^{n}}}\widehat{\xi}\dd\left(\widehat{\mu}+\widehat{\nu}\right)
\]
 for all cosmically continuous function $\xi:\RR^{n}\to\RR$. By Proposition
\ref{prop:cosmic-characterization} this equivalent to saying that
\[
\int_{\overline{\SS_{-}^{n}}}\rho\dd\left(\widehat{\mu_{k}}+\widehat{\nu_{k}}\right)\to\int_{\overline{\SS_{-}^{n}}}\rho\dd\left(\widehat{\mu}+\widehat{\nu}\right)
\]
 for all continuous functions $\rho:\overline{\SS_{-}^{n}}\to\RR$,
which precisely means that $\widehat{\mu_{k}}+\widehat{\nu_{k}}\to\widehat{\mu}+\widehat{\nu}$
weakly on $\overline{\SS_{-}^{n}}$. 
\end{proof}
Given Proposition \ref{prop:cosmic-weak}, it is natural to ask for
an explicit description of the measure $\widehat{\mu_{f}}+\widehat{\nu_{f}}$
for $f=e^{-\phi}\in\cvx$. Recall that the domain of a function $\phi\in\cvx$
is 
\[
\dom(\phi)=\left\{ x\in\RR^{n}:\ \phi(x)<\infty\right\} \subseteq\RR^{n},
\]
 and the epigraph of $\phi$ is 
\[
\epi(\phi)=\left\{ (x,t)\in\RR^{n+1}:\ x\in\dom(\phi),\ t\ge\phi(x)\right\} \subseteq\RR^{n+1}.
\]
As usual, we denote by $\HH^{n}$ the $n$-dimensional Hausdorff measure.
The epigraph $\epi\left(\phi\right)$ is an (unbounded) closed convex
set with non-empty interior, so the Gauss map $n_{\epi(\phi)}$ is
defined $\HH^{n}$-almost everywhere. Moreover, for every $(x,t)\in\partial\epi(\phi)$
it is clear that $(x,t')\in\epi(\phi)$ for all $t'\ge t$, which
implies that $n_{\epi(\phi)}(x,t)\in\overline{\SS_{-}^{n}}$ . We
can now state the result:
\begin{prop}
\label{prop:boundary-repr}For every $f=e^{-\phi}\in\lc$ we have
\[
\widehat{\mu_{f}}+\widehat{\nu_{f}}=\left(n_{\epi(\phi)}\right)_{\sharp}\left(\left.e^{-t}\dd\HH^{n}(x,t)\right|_{\partial\epi(\phi)}\right).
\]
 Explicitly, this means that for all continuous functions $\rho:\overline{\SS_{-}^{n}}\to\RR$
one has 
\[
\int_{\overline{\SS_{-}^{n}}}\rho\dd\left(\widehat{\mu_{f}}+\widehat{\nu_{f}}\right)=\int_{\partial\epi(\phi)}\rho\left(n_{\epi(\phi)}(x,t)\right)e^{-t}\dd\HH^{n}(x,t).
\]
\end{prop}

\begin{proof}
Write $\rho=\widehat{\xi}$ for a cosmically continuous function $\xi:\RR^{n}\to\RR$.
We partition $\partial\epi(\phi)$ as $\partial\epi(\phi)=A\cup B$
for 
\[
A=\left\{ (x,\phi(x)):\ x\in\dom(\phi)\right\} 
\]
 and 
\[
B=\left\{ (x,t):\ x\in\partial\dom(\phi),\ t\ge\phi(x)\right\} .
\]
 We first prove that 
\begin{equation}
\int_{\SS_{-}^{n}}\widehat{\xi}\dd\widehat{\mu_{f}}=\int_{A}\widehat{\xi}\left(n_{\epi(\phi)}(x,t)\right)e^{-t}\dd\HH^{n}(x,t).\label{eq:boundary-A}
\end{equation}
 Indeed, by definition we have 
\[
\int_{\SS_{-}^{n}}\widehat{\xi}\dd\widehat{\mu_{f}}=\int_{\RR^{n}}\xi\dd\mu_{f}=\int_{\dom(\phi)}\xi\left(\nabla\phi(u)\right)e^{-\phi(u)}\dd u.
\]
 Perform the (locally Lipschitz) change of variables $T:\dom\left(\phi\right)\to A$
defined by $Tu=(u,\phi(u))$. The Jacobian of $T$ is 
\[
JT(u)=\det\left(DT(u)^{t}\cdot DT(u)\right)=\sqrt{\det\left(Id+\nabla\phi\otimes\nabla\phi\right)}=\sqrt{1+\left|\nabla\phi(u)\right|^{2}},
\]
 so by the change of variables formula for Lipschitz maps (see \cite[Theorem 3.9]{EvansGariepyMeasure2015})
\begin{equation}
\int_{\SS_{-}^{n}}\widehat{\xi}\dd\widehat{\mu_{f}}=\int_{A}\xi\left(\nabla\phi(x)\right)e^{-t}\cdot\frac{1}{\sqrt{1+\left|\nabla\phi(x)\right|^{2}}}\dd\HH^{n}(x,t).\label{eq:boundary-A-middle}
\end{equation}
 But for almost all $(x,t)\in A$ we have $n_{\epi(\phi)}(x,t)=\frac{\left(\nabla\phi(x),-1\right)}{\left|\left(\nabla\phi(x),-1\right)\right|}$.
Therefore 
\begin{equation}
\widehat{\xi}\left(n_{\epi(\phi)}(x,t)\right)=\widehat{\xi}\left(\frac{\left(\nabla\phi(x),-1\right)}{\left|\left(\nabla\phi(x),-1\right)\right|}\right)=\frac{\xi(\nabla\phi(x))}{\sqrt{1+\left|\nabla\phi(x)\right|^{2}}},\label{eq:boundary-hat}
\end{equation}
 where we used formula (\ref{eq:hat-to-nohat}). Equations (\ref{eq:boundary-A-middle})
and (\ref{eq:boundary-hat}) together clearly imply (\ref{eq:boundary-A}). 

We now prove that 
\begin{equation}
\int_{\overline{\SS_{-}^{n}}\setminus\SS_{-}^{n}}\widehat{\xi}\dd\widehat{\nu_{f}}=\int_{B}\widehat{\xi}\left(n_{\epi(\phi)}(x,t)\right)e^{-t}\dd\HH^{n}(x,t).\label{eq:boundary-B}
\end{equation}
 Indeed, for almost every point $(x,t)\in B$ we have $n_{\epi(\phi)}(x,t)=(n_{\dom(\phi)}(x),0)$.
Using this fact, Fubini's theorem, and the connection between $\widehat{\xi}$
and $\overline{\xi}$ from Proposition \ref{prop:cosmic-characterization}
we have 
\begin{align*}
\int_{B}\widehat{\xi}\left(n_{\epi(\phi)}\left(x,t\right)\right)e^{-t}\dd\HH^{n}(x,t) & =\int_{x\in\partial\dom(\phi)}\int_{t=\phi(x)}^{\infty}\widehat{\xi}\left(n_{\epi(\phi)}\left(x,t\right)\right)e^{-t}\dd t\dd\HH^{n-1}(x)\\
 & =\int_{x\in\partial\dom(\phi)}\int_{t=\phi(x)}^{\infty}\widehat{\xi}\left(n_{\dom(\phi)}(x),0\right)e^{-t}\dd t\dd\HH^{n-1}(x)\\
 & =\int_{x\in\partial\dom(\phi)}\overline{\xi}\left(n_{\dom(\phi)}(x)\right)e^{-\phi(x)}\dd\HH^{n-1}(x)\\
 & =\int_{\SS^{n-1}}\overline{\xi}\dd\nu_{f}=\int_{\overline{\SS_{-}^{n}}\setminus\SS_{-}^{n}}\widehat{\xi}\dd\widehat{\nu_{f}}
\end{align*}
 as claimed.

Since $\HH^{n}(A\cap B)=0$, the identities (\ref{eq:boundary-A})
and (\ref{eq:boundary-B}) together imply the result.
\end{proof}
We record for later use that what we've actually shown is that for
every cosmically continuous function $\xi:\RR^{n}\to\RR$ one has
\begin{equation}
\int_{\RR^{n}}\xi\dd\mu_{f}+\int_{\SS^{n-1}}\overline{\xi}\dd\nu_{f}=\int_{\partial\epi(\phi)}\widehat{\xi}\left(n_{\epi(\phi)}(x,t)\right)e^{-t}\dd\HH^{n}(x,t).\label{eq:var-alternative}
\end{equation}

Using equation (\ref{eq:var-alternative}) we can rewrite the first
variation formula (\ref{eq:representation}) in the following nice
way:
\begin{cor}
For every $f=e^{-\phi}\in\lc$ and an upper semicontinuous log-concave
function $g:\RR^{n}\to\RR$ we have
\[
\delta(f,g)=\int_{\partial\epi(\phi)}\widehat{h_{g}}\left(n_{\epi(\phi)}(x,t)\right)e^{-t}\dd\HH^{n}(x,t).
\]
 
\end{cor}

Here we define $\widehat{h_{g}}:\overline{\SS_{-}^{n}}\to(-\infty,\infty]$
in the natural way: $h_{g}:\RR^{n}\to(-\infty,\infty]$ is usually
not cosmically continuous, but we may still define $\widehat{h_{g}}$
on $\SS_{-}^{n}$ by formula (\ref{eq:nohat-to-hat}), and then on
the equator $\overline{\SS_{-}^{n}}\setminus\SS_{-}^{n}$ define $\widehat{h_{g}}(\theta,0)=\overline{h_{g}}(\theta)$.
Even though $h_{g}$ is not cosmically continuous, the proof that
\[
\int h_{g}\dd\mu_{f}+\int\overline{h_{g}}\dd\nu_{f}=\int_{\partial\epi(\phi)}\widehat{h_{g}}\left(n_{\epi(\phi)}(x,t)\right)e^{-t}\dd\HH^{n}(x,t)
\]
still works in exactly the same way as the proof of (\ref{eq:var-alternative}).
The corollary follows immediately from this identity and (\ref{eq:representation}). 

\section{\label{sec:continuity}Continuity of functional surface area measures}

In this section we prove Theorem \ref{thm:measures-cont}. Fix $\left\{ f_{k}\right\} _{k=1}^{\infty},f\in\lc$
such that $f_{k}\to f$. Write as usual $f_{k}=e^{-\phi_{k}}$ and
$f=e^{-\phi}$. Our goal is to prove that $\left(\mu_{f_{k}},\nu_{f_{k}}\right)\to\left(\mu_{f},\nu_{f}\right)$
cosmically. By (\ref{eq:var-alternative}), we need to prove that
for every cosmically continuous function $\xi:\RR^{n}\to\RR$ one
has 
\[
\int_{\partial\epi(\phi_{k})}\widehat{\xi}\left(n_{\epi(\phi_{k})}(x,t)\right)e^{-t}\dd\HH^{n}(x,t)\to\int_{\partial\epi(\phi)}\widehat{\xi}\left(n_{\epi(\phi)}(x,t)\right)e^{-t}\dd\HH^{n}(x,t).
\]

The main idea of the proof is to perform a change of variables and
transform the domain of integration to $\RR^{n}$, using a new notion
of a \emph{curvilinear radial function} for convex functions. For
technical reasons we will need to work with the following classes
of convex functions:
\begin{defn}
For every $\epsilon>0$ we set 
\[
\cvx^{(\epsilon)}=\left\{ \phi\in\cvx:\ \phi(x)<\min\phi+\frac{1}{2}\text{ for all }x\in\RR^{n}\text{ with }\left|x\right|\le\epsilon\right\} .
\]
\end{defn}

The choice of the constant $\frac{1}{2}$ in the definition of $\cvx^{(\epsilon)}$
is immaterial, but choosing a constant smaller than $1$ helps to
eliminate some other constants later in the proof. The following rather
technicals properties of the classes $\cvx^{(\epsilon)}$ will be
crucial for the proof:
\begin{lem}
\label{lem:cvx-eps-properties}
\begin{enumerate}
\item \label{enu:cvx-eps-trans}For every $\phi\in\cvx$ there exists $v\in\RR^{n}$
and $\epsilon>0$ such that $\tau_{v}\phi\in\cvx^{(\epsilon)}$, where
$\left(\tau_{v}\phi\right)(x)=\phi(x+v)$. 
\item \label{enu:cvx-eps-open}Assume that $\left\{ \phi_{k}\right\} _{k=1}^{\infty}\subseteq\cvx$
and $\phi_{k}\to\phi\in\cvx^{(\epsilon)}$ for some $\epsilon>0$.
Then $\phi_{k}\in\cvx^{(\epsilon/2)}$ for all large enough $k$. 
\item \label{enu:cvx-eps-bound}Assume $\phi\in\cvx^{(\epsilon)}$. Fix
$(x_{0},t_{0})\in\partial\epi(\phi)$ and let $v=(v_{x},v_{t})\in\RR^{n+1}$
be an outer normal to $\epi(\phi)$ at $(x_{0},t_{0})$. Then 
\[
\left\langle (x_{0},-1),v\right\rangle =\left\langle x_{0},v_{x}\right\rangle -v_{t}\ge\min\left(\epsilon,\frac{1}{2}\right)\left|v\right|^{2}.
\]
\end{enumerate}
\end{lem}

\begin{proof}
\begin{enumerate}
\item Fix $w\in\RR^{n}$ such that $\phi(w)=\min\phi$. If $w\in\inte(\dom(\phi))$
we just choose $v=w$. If $w\in\partial\dom(\phi)$, then every $v\in\inte(\dom(\phi))$
close enough to $w$ will satisfy
\[
\phi(v)\le\phi(w)+\frac{1}{4}=\min\phi+\frac{1}{4}.
\]
 Since $\phi$ is continuous on $\inte(\dom(\phi))$ there exists
$\epsilon>0$ such that for all $x$ with $\left|x-v\right|\le\epsilon$
we have 
\[
\phi(x)<\phi(v)+\frac{1}{4}\le\min\phi+\frac{1}{2}.
\]
 This shows that $\tau_{v}\phi\in\cvx^{(\epsilon)}$.
\item Assume that $\left\{ \phi_{k}\right\} _{k=1}^{\infty}\subseteq\cvx$
and $\phi_{k}\to\phi\in\cvx^{(\epsilon)}$. Since $\overline{B}_{\epsilon}(0)\subseteq\dom(\phi)$
we have $\overline{B}_{\epsilon/2}(0)\subseteq\inte\left(\dom(\phi)\right)$
, and it follows that $\phi_{k}\to\phi$ uniformly on $\overline{B}_{\epsilon/2}(0)$
-- See \cite[Theorem 7.17(c)]{RockafellarWetsVariational1998}. Also
by \cite[Lemma 12]{ColesantiEtAlValuations2019} we have $\min\phi_{k}\to\min\phi$.
Therefore
\[
\lim_{k\to\infty}\left(\max_{\overline{B}_{\epsilon/2}(0)}\phi_{k}\right)=\max_{\overline{B}_{\epsilon/2}(0)}\phi<\min\phi+\frac{1}{2}=\lim_{k\to\infty}\phi_{k}+\frac{1}{2}.
\]
 It follows that indeed $\max_{\overline{B}_{\epsilon/2}(0)}\phi_{k}<\phi_{k}+\frac{1}{2}$
for all large enough $k$. 
\item Recall that we always have $v_{t}\le0$. If $v_{x}=0$ then
\[
\left\langle x_{0},v_{x}\right\rangle -v_{t}=\left|v_{t}\right|=\left|v\right|
\]
 and there is noting to prove. Otherwise, since $\epsilon\frac{v_{x}}{\left|v_{x}\right|}\in\overline{B}_{\epsilon}(0)$
we know that $\phi\left(\epsilon\frac{v_{x}}{\left|v_{x}\right|}\right)\le\min\phi+\frac{1}{2}$
, and so $(\epsilon\frac{v_{x}}{\left|v_{x}\right|},\min\phi+\frac{1}{2})\in\epi(\phi)$.
It follows that 
\[
\left\langle \left(\epsilon\frac{v_{x}}{\left|v_{x}\right|},\min\phi+\frac{1}{2}\right)-\left(x_{0},t_{0}\right),(v_{x},v_{t})\right\rangle \le0.
\]
 But $t_{0}\ge\phi(x_{0})\ge\min\phi$, and so 
\begin{align*}
\left\langle \left(\epsilon\frac{v_{x}}{\left|v_{x}\right|},\min\phi+\frac{1}{2}\right)-\left(x_{0},t_{0}\right),(v_{x},v_{t})\right\rangle  & =\left\langle \epsilon\frac{v_{x}}{\left|v_{x}\right|}-x_{0},v_{x}\right\rangle +\left(\min\phi+\frac{1}{2}-t_{0}\right)v_{t}\\
 & \ge\epsilon\left|v_{x}\right|-\left\langle x_{0},v_{x}\right\rangle +\frac{1}{2}v_{t}.
\end{align*}
 Therefore $\epsilon\left|v_{x}\right|-\left\langle x_{0},v_{x}\right\rangle +\frac{1}{2}v_{t}\le0$,
or 
\[
\left\langle x_{0},v_{x}\right\rangle -v_{t}\ge\epsilon\left|v_{x}\right|+\frac{1}{2}\left|v_{t}\right|\ge\min\left(\epsilon,\frac{1}{2}\right)\cdot\sqrt{\left|v_{x}\right|^{2}+v_{t}^{2}}=\min\left(\epsilon,\frac{1}{2}\right)\left|v\right|^{2}
\]
 as claimed. 
\end{enumerate}
\end{proof}
Following part (\ref{enu:cvx-eps-trans}) of the lemma, we remark
that if $\phi$ attains its minimum on $\inte\left(\dom\left(\phi\right)\right)$
then after translation one can assume not only that $\phi\in\cvx^{(\epsilon)}$
but also that $\phi(0)=\min\phi$. Using this property can greatly
simplify some of the arguments below such as the proof of Proposition
\ref{prop:radial-def}. Unfortunately, this simplification is not
possible if $\phi$ attains its minimum (only) on $\partial\dom(\phi)$. 

The following proposition establishes the main definition for this
section: 
\begin{prop}
\label{prop:radial-def}Fix $\phi\in\cvx^{(\epsilon)}$ and $u\in\RR^{n}$.
Consider the curve $\gamma_{u}:(0,\infty)\to\RR^{n}\times\RR$ defined
by $\gamma_{u}(s)=\left(su,\log\left(\frac{1}{s}\right)\right)$.
Then there exists a number $s_{\phi}(u)\in(0,\infty)$ with the following
properties:
\begin{enumerate}
\item $\gamma_{u}(s)\in\inte(\epi(\phi))$ for all $s<s_{\phi}(u)$. 
\item $\gamma_{u}\left(s_{\phi}(u)\right)\in\partial\epi(\phi)$.
\item $\gamma_{u}(s)\notin\epi(\phi)$ for all $s>s_{\phi}(u)$. 
\end{enumerate}
The value $s_{\phi}(u)$ is clearly unique. The function $s_{\phi}:\RR^{n}\to(0,\infty)$
mapping $u$ to $s_{\phi}(u)$ is called the \emph{curvilinear radial
function} of $\phi$.
\end{prop}

\begin{proof}
Since $\phi\in\cvx^{(\epsilon)}$ we know in particular that $\phi$
is bounded on $\overline{B}_{\epsilon}(0)$. Since $\gamma_{u}(s)\xrightarrow{s\to0^{+}}\left(0,+\infty\right)$,
we clearly have $\gamma_{u}(s)\in\inte\left(\epi(\phi)\right)$ for
small enough $s$. On the other hand, since $\phi$ is bounded from
below and $\log\frac{1}{s}\xrightarrow{s\to\infty}-\infty$, we have
$\gamma_{u}(s)\notin\epi(\phi)$ for large enough $s$. By continuity,
there exists $s_{0}\in(0,\infty)$ such that $\gamma_{u}(s_{0})\in\partial\epi(\phi)$.
We write $\gamma_{u}(s_{0})=(x_{0},t_{0})$. 

Our next claim is that there exists $\delta>0$ such that $\gamma_{u}(s_{0}-s)\in\inte\left(\epi(\phi)\right)$
and $\gamma_{u}(s_{0}+s)\notin\epi(\phi)$ for all $0<s\le\delta$.
To prove the claim, it is enough to show that for every outer normal
$v=(v_{x},v_{t})\in\RR^{n+1}$ to $\epi(\phi)$ at $\gamma_{u}(s_{0})$
we have $\left\langle \gamma_{u}^{\prime}(s_{0}),v\right\rangle >0$.
But $\gamma_{u}^{\prime}(s_{0})=\left(u,-\frac{1}{s_{0}}\right)$,
so 
\[
\left\langle \gamma_{u}^{\prime}(s_{0}),v\right\rangle =\frac{1}{s_{0}}\left\langle \left(s_{0}u,-1\right),v\right\rangle =\frac{1}{s_{0}}\left\langle (x_{0},-1),v\right\rangle \ge\frac{1}{s_{0}}\min\left(\epsilon,\frac{1}{2}\right)\left|v\right|^{2}>0
\]
 by Lemma \ref{lem:cvx-eps-properties}(\ref{enu:cvx-eps-bound})
and the claim is proved. 

The rest of the proof is a straightforward topological argument. Assume
by contradiction that there exists a point $s_{1}\ne s_{0}$ such
that $\gamma_{u}(s_{1})\in\partial\epi(\phi)$. If $s_{1}<s_{0}$
then the set
\[
\left\{ s\in[s_{1},s_{0}-\delta]:\ \gamma_{u}(s)\in\partial\epi(\phi)\right\} 
\]
is a non-empty compact set, so it has some maximum $s_{2}$. But then
the same argument as above shows that $\gamma_{u}(s_{2}+\delta')\notin\epi(\phi)$
for small enough $\delta'>0$, and since $\gamma_{u}(s_{0}-\delta)\in\inte(\epi(\phi))$
there must exist a point $s_{3}\in(s_{2}+\delta',s_{0}-\delta)$ such
that $\gamma_{u}(s_{3})\in\partial\epi(\phi)$, contradicting the
maximality of $s_{2}$. The case $s_{1}>s_{0}$ is the same. It follows
that $s_{0}$ is the unique point which satisfies $\gamma_{u}(s_{0})\in\epi(\phi)$. 

If now $\gamma_{u}(s)\notin\inte(\epi(\phi))$ for $s<s_{0}$, then
since $\gamma_{u}(s_{0}-\delta)\in\inte(\epi(\phi))$ again by continuity
we could find $s\le\widetilde{s}<s_{0}-\delta$ such that $\gamma_{u}(\widetilde{s})\in\partial\epi(\phi)$,
which is a contradiction. The same argument shows that $\gamma_{u}(s)\notin\epi(\phi)$
for all $s>s_{0}$, completing the proof. 
\end{proof}
From the function $s_{\phi}$ we obtain our desired parametrization
of $\partial\epi(\phi)$: 
\begin{cor}
\label{cor:epi-param}For $\phi\in\cvx^{(\epsilon)}$ define $F_{\phi}:\RR^{n}\to\partial\epi(\phi)$
by 
\[
F_{\phi}(u)=\gamma_{u}\left(s_{\phi}(u)\right)=\left(s_{\phi}(u)u,\log\left(\frac{1}{s_{\phi}(u)}\right)\right).
\]
 Then $F_{\phi}$ is a bijection between $\RR^{n}$ and $\partial\epi(\phi)$,
with the inverse map $G_{\phi}:\partial\epi(\phi)\to\RR^{n}$ given
by $G_{\phi}(x,t)=e^{t}x$. 
\end{cor}

\begin{proof}
The fact that $F_{\phi}(u)\in\partial\epi(\phi)$ follows from Proposition
\ref{prop:radial-def}. It is immediate from the formulas that $G_{\phi}\left(F_{\phi}(u)\right)=u$
for all $u\in\RR^{n}$. Conversely, for every $(x,t)\in\partial\epi(\phi)$
we have 
\[
\gamma_{G_{\phi}(x,t)}\left(e^{-t}\right)=\left(x,t\right)\in\partial\epi(\phi),
\]
Which precisely means by definition that $s_{\phi}\left(G_{\phi}(x,t)\right)=e^{-t}$
and $F_{\phi}\left(G_{\phi}(x,t)\right)=(x,t)$. This proves the claim. 
\end{proof}
\begin{rem}
Assume $\phi\in\cvx$ satisfies $\min\phi=\phi(0)=0$. In \cite{Artstein-AvidanMilmanHidden2011},
Artstein-Avidan and Milman defined the $\mathcal{J}$-transform of
such functions, which can be considered as their Minkowski functional.
Therefore the reciprocal $\frac{1}{\mathcal{J}\phi}$ can be thought
of as a radial function for $\phi$. Explicitly we have
\[
\frac{1}{\mathcal{J}\phi}(u)=\sup\left\{ s\ge0:\ (su,s)\in\epi(\phi)\right\} ,
\]
 i.e. $\frac{1}{\mathcal{J}\phi}(u)$ is the radial function of $\epi(\phi)$
in the direction $(u,1)$. This is similar to our definition, with
the non-linear curve $\gamma_{u}(s)$ replaced by the linear curve
$\widetilde{\gamma}_{u}(s)=(su,s)$. In many ways the function $\frac{1}{\mathcal{J}\phi}$
is simpler and more canonical than our function $s_{\phi}$. For example
$\mathcal{J}\phi$ is a convex function, while $s_{\phi}$ does not
seem to have any convexity properties. However, unlike Corollary \ref{cor:epi-param},
the corresponding map 
\[
\widetilde{F}_{\phi}(s)=\widetilde{\gamma}_{u}\left(\frac{1}{\mathcal{J}\phi}(u)\right)=\frac{(u,1)}{\left(\mathcal{J}\phi\right)(u)}
\]
 does not always parametrize the boundary $\partial\epi(\phi)$. As
an example, for $\phi(x)=\left|x\right|$ we have 
\[
\left(\mathcal{J}\phi\right)(u)=\oo_{B_{2}^{n}}^{\infty}(u)=\begin{cases}
0 & \left|u\right|\le1\\
\infty & \left|u\right|>1,
\end{cases}
\]
 so $\widetilde{F}_{\phi}$ is not even well-defined and in any case
does not parametrize any part of $\partial\epi(\phi)$. Therefore
the curvilinear radial function $s_{\phi}$ is more suitable for our
purpose.
\end{rem}

Our next goal is prove that $s_{\phi}$ is a locally Lipschitz function
and compute its gradient, which we will do using the implicit function
theorem. Since $\phi$ is not necessarily smooth we need a version
of the implicit function theorem for Lipschitz functions, which is
due to Clarke (\cite{ClarkeOptimization1990}). We give the necessary
definitions:

Consider a (locally) Lipschitz function $\Phi:\RR^{N}\to\RR$. The
generalized gradient of $\Phi$ at a point $x_{0}$ can be defined
by (\cite[Theorem 2.5.1]{ClarkeOptimization1990})
\[
\partial\Phi(x_{0})=\conv\left\{ \lim_{i\to\infty}\nabla\Phi(x_{i}):\ \begin{array}{l}
x_{i}\to x,\ \Phi\text{ is differentiable at every }x_{i}\\
\text{and }\lim_{i\to\infty}\nabla\Phi(x_{i})\text{ exists.}
\end{array}\right\} .
\]
If $\Phi$ is convex, the generalized gradient agrees with the usual
subgradient (\cite[Proposition 2.2.7]{ClarkeOptimization1990}). If
we have a map $\Phi:\RR^{N}\times\RR^{M}\to\RR$, which we write as
$\Phi(x,y)$, we denote by $\partial_{y}\Phi(x_{0},y_{0})$ the generalized
gradient only in the $y$ variable. Formally these are all vectors
$v\in\RR^{M}$ such that $\left(w\ |\ v\right)\in\partial\Phi(x_{0},y_{0})$
for some $w\in\RR^{N}$. Generalized gradients satisfy the chain rule
(\cite[Theorem 2.3.10]{ClarkeOptimization1990}): If $\Psi:\RR^{N}\to\RR^{N}$
is smooth and $\Phi:\RR^{N}\to\RR$ is Lipschitz then 
\[
\partial\left(\Phi\circ\Psi\right)(x_{0})\subseteq\left\{ v\cdot D\Psi(x_{0}):\ v\in\partial\Phi\left(\Psi(x_{0})\right)\right\} 
\]
 (in many cases including ours there is actually an equality, but
we will not need this fact). The implicit function theorem then reads:
\begin{thm}[{\cite[page 256]{ClarkeOptimization1990}}]
Assume $\Phi:\RR^{N}\times\RR\to\RR$ is a Lipschitz function and
$\Phi(x_{0},y_{0})=0$. Assume further that $0\notin\partial_{y}\Phi(x_{0},y_{0})$.
Then there exists a neighborhood $U$ of $x_{0}$ and a \textbf{Lipschitz}
function $\Psi:U\to\RR^ {}$ such $\Psi(x_{0})=y_{0}$ and $\Phi(x,\Psi(x))=0$
for all $x\in U$. 
\end{thm}

Of course a similar theorem holds for functions $\Phi:\RR^{N}\times\RR^{M}\to\RR^{M}$,
but this version will suffice for our goals. Using these tools we
prove:
\begin{prop}
For all $\phi\in\cvx^{(\epsilon)}$ the function $s_{\phi}$ is locally
Lipschitz. Moreover, for almost every $u_{0}\in\RR^{n}$ we have 
\begin{equation}
\nabla s_{\phi}(u_{0})=-s_{\phi}(u_{0})^{2}\cdot\frac{n_{\epi(\phi)}^{(x)}(x_{0},t_{0})}{\left\langle n_{\epi(\phi)}(x_{0},t_{0}),(x_{0},-1)\right\rangle }.\label{eq:radial-grad}
\end{equation}
 Here $(x_{0},t_{0})=F_{\phi}(u_{0})\in\partial\epi(\phi)$ and $n_{\epi(\phi)}^{(x)}$
denotes the first $n$ coordinates of the outer unit normal $n_{\epi(\phi)}$. 
\end{prop}

\begin{proof}
Fix an arbitrary point $(\widetilde{x},\widetilde{t})\in\inte\left(\epi(\phi)\right)$,
and let $\rho:\RR^{n+1}\to\RR$ be the Minkowski functional of $\epi(\phi)$
with respect to $(\widetilde{x},\widetilde{t})$: 
\[
\rho(x,t)=\inf\left\{ \lambda>0:\frac{(x,t)-(\widetilde{x},\widetilde{t})}{\lambda}\in\epi(\phi)\right\} .
\]
 Then $\rho$ is convex, and $\rho(x,t)=1$ if and only if $(x,t)\in\partial\epi(\phi)$.
Moreover, at every $(x,t)\in\partial\epi(\phi)$ the elements of the
subgradient $\partial\rho(x,t)$ are outer normals to $\epi(\phi)$
at $(x,t)$. 

Consider now the function $\Phi:\RR^{n}\times\RR_{+}\to\RR$ defined
by $\Phi(u,s)=\rho\left(su,\log\left(\frac{1}{s}\right)\right)$.
Fix a point $u_{0}\in\RR^{n}$ and set $s_{0}=s_{\phi}(u_{0})$. Then
by definition we have $\Phi(u_{0},s_{0})=\rho\left(F_{\phi}(u_{0})\right)=1$.
Moreover, by the chain rule we have 
\begin{align*}
\partial_{s}\Phi(u_{0},s_{0}) & \subseteq\left\{ \left\langle v_{x},u_{0}\right\rangle -\frac{v_{t}}{s_{0}}:\ (v_{x},v_{t})\in\partial\rho\left(s_{0}u_{0},\log\left(\frac{1}{s_{0}}\right)\right)\right\} \\
 & =\left\{ \left\langle v_{x},u_{0}\right\rangle -\frac{v_{t}}{s_{0}}:\ (v_{x},v_{t})\in\partial\rho\left(F_{\phi}(u_{0})\right)\right\} 
\end{align*}
 Define $(x_{0},t_{0})=F_{\phi}(u_{0})\in\epi(\phi)$. Since every
vector $(v_{x},v_{t})\in\partial\rho(x_{0},t_{0})$ is an outer normal
to $\epi(\phi)$ at $(x_{0},t_{0})$, it follows from Lemma \ref{lem:cvx-eps-properties}(\ref{enu:cvx-eps-bound})
that 
\[
\left\langle v_{x},u_{0}\right\rangle -\frac{v_{t}}{s_{0}}=\frac{1}{s_{0}}\left(\left\langle v_{x},x_{0}\right\rangle -v_{t}\right)\ge\frac{1}{s_{0}}\cdot\min\left(\epsilon,\frac{1}{2}\right)\left|v\right|^{2}>0.
\]
Therefore by the implicit function theorem the equation $\Phi(u,s)=1$
defines $s$ as a Lipschitz function of $u$ locally around $u_{0}$.
But this function is exactly $s_{\phi}$, so $s_{\phi}$ is locally
Lipschitz. 

To compute $\nabla s_{\phi}$, fix a point $u_{0}\in\RR^{n}$ such
that $s_{\phi}$ is differentiable at $u_{0}$ and $\epi(\phi)$ has
a unique outer unit normal at $(x_{0},t_{0})=F_{\phi}(u_{0})$ --
these conditions hold almost everywhere since $s_{\phi}$ is locally
Lipschitz and $\epi(\phi)$ is convex. Then $\rho$ is differentiable
at $(x_{0},t_{0})$ and $\nabla\rho(x_{0},t_{0})=(v_{x},v_{t})=c\cdot n_{\epi(\phi)}(x_{0},t_{0})$
for some $c>0$. Differentiating the identity $\Phi(u,s_{\phi}(u))=1$
at $u_{0}$ using the (standard) chain rule we obtain 
\begin{align*}
\nabla s_{\phi}(u_{0}) & =-\frac{\left(D_{u}\Phi\right)(u_{0},s_{\phi}(u_{0}))}{\left(D_{s}\Phi\right)(u_{0},s_{\phi}(u_{0}))}=-\frac{s_{\phi}(u_{0})\cdot v_{x}}{\left\langle v_{x},u_{0}\right\rangle -\frac{v_{t}}{s_{\phi}(u_{0})}}\\
 & =-s_{\phi}(u_{0})^{2}\frac{v_{x}}{\left\langle (v_{x},v_{t}),\left(s_{\phi}(u_{0})u_{0},-1\right)\right\rangle }=-s_{\phi}(u_{0})^{2}\cdot\frac{n_{\epi(\phi)}^{(x)}(x_{0},t_{0})}{\left\langle n_{\epi(\phi)}(x_{0},t_{0}),(x_{0},-1)\right\rangle }
\end{align*}
as claimed. 
\end{proof}
Our next goal is to understand the behavior of the curvilinear radial
function under epi-convergence:
\begin{prop}
\label{prop:radial-continuous}Assume that $\left\{ \phi_{k}\right\} _{k=1}^{\infty},\phi\in\cvx^{(\epsilon)}$
for some $\epsilon>0$ and that $\phi_{k}\to\phi$. Then:
\begin{enumerate}
\item For all $u\in\RR^{n}$ we have $s_{\phi_{k}}(u)\to s_{\phi}(u)$ and
$F_{\phi_{k}}(u)\to F_{\phi}(u)$. 
\item For almost every $u\in\RR^{n}$ we have $\nabla s_{\phi_{k}}(u)\to\nabla s_{\phi}(u)$
and $DF_{\phi_{k}}(u)\to DF_{\phi}(u)$. 
\end{enumerate}
\end{prop}

\begin{proof}
\begin{enumerate}
\item Since $\left\{ \phi_{k}\right\} _{k=1}^{\infty}$ is convergent it
is uniformly coercive, i.e. $\phi_{k}\ge\psi_{L}$ for some function
of the form $\psi_{L}(x)=a\left|x\right|+b$. We also know that for
all $x\in\overline{B}_{\epsilon}(0)$ we have 
\[
\phi_{k}(x)\le\min\phi_{k}+\frac{1}{2}\rightarrow\min\phi+\frac{1}{2},
\]
 so the functions $\left\{ \phi_{k}\right\} _{k=1}^{\infty}$ are
uniformly bounded on $\overline{B}_{\epsilon}(0)$. Therefore $\phi_{k}\le\psi_{U}$
for some function $\psi_{U}$ of the form $\psi_{U}=\oo_{\overline{B}_{\epsilon}(0)}^{\infty}+M$.
Therefore $0<s_{\psi_{U}}(u)\le s_{\phi_{k}}(u)\le s_{\psi_{L}}(u)<\infty$
for all $u$, i.e. the sequence $\left\{ s_{\phi_{k}}(u)\right\} _{k=1}^{\infty}$
is bounded. It is therefore enough to prove that for every converging
subsequence $s_{\phi_{k_{\ell}}}(u)\to\widetilde{s}$ we must have
$\widetilde{s}=s_{\phi}(u)$. 

We have $\gamma_{u}\left(s_{\phi_{k_{\ell}}}(u)\right)\in\epi\left(\phi_{k_{\ell}}\right)$
and $\gamma_{u}\left(s_{\phi_{k_{\ell}}}(u)\right)\to\gamma_{u}\left(\widetilde{s}\right)$.
Since $\epi(\phi_{k_{\ell}})\to\epi(\phi)$, it follows that $\gamma_{u}(\widetilde{s})\in\epi(\phi)$.
Therefore $\widetilde{s}\le s_{\phi}(u)$. 

Assume by contradiction that $\widetilde{s}<s_{\phi}(u)$, so $\gamma_{u}\left(\widetilde{s}\right)\in\inte\left(\epi(\phi)\right)$.
Choose a closed ball $\overline{B}$ around $\gamma_{u}\left(\widetilde{s}\right)$
such that $\overline{B}\subseteq\inte\left(\epi(\phi)\right)$. Since
$\gamma_{u}\left(s_{\phi_{k_{\ell}}}(u)\right)\to\gamma_{u}\left(\widetilde{s}\right)$
we clearly have $\gamma_{u}\left(s_{\phi_{k_{\ell}}}(u)\right)\in\overline{B}$
for all large enough $\ell$. But $\epi(\phi_{k_{\ell}})\to\epi(\phi)$,
so by \cite[Proposition 4.15]{RockafellarWetsVariational1998}, we
know that $\overline{B}\subseteq\inte\left(\epi(\phi_{k_{\ell}})\right)$
for all large enough $\ell$. This means that for large enough $\ell$
we have $\gamma_{u}\left(s_{\phi_{k_{\ell}}}(u)\right)\in\inte\left(\epi(\phi_{k_{\ell}})\right)$,
which is impossible. Therefore $\widetilde{s}=s_{\phi}(u)$ and the
proof is complete. 

The fact that we also have $F_{\phi_{k}}(u)\to F_{\phi}(u)$ is now
immediate from its definition. 
\item Fix $u$ such that all the functions $s_{\phi_{k}}$ differentiable
at $u$, every $\epi(\phi_{k})$ has a unique outer unit normal at
$F_{\phi_{k}}(u)$, and the same holds for $s_{\phi}$ and $\epi(\phi)$.
We first claim that since $\epi(\phi_{k})\to\epi(\phi)$ and $F_{\phi_{k}}(u)\to F_{\phi}(u)$
we must have 
\begin{equation}
n_{\epi(\phi_{k})}\left(F_{\phi_{k}}(u)\right)\to n_{\epi(\phi)}\left(F_{\phi}(u)\right).\label{eq:normal-converge}
\end{equation}
Indeed, choose any converging subsequence $n_{\epi(\phi_{k_{\ell}})}\left(F_{\phi_{k_{\ell}}}(u)\right)\to\widetilde{n}$.
Every $(x,t)\in\inte\left(\epi(\phi)\right)$ also belongs to $\epi(\phi_{k})$
for all large enough $k$, and therefore 
\[
\left\langle (x,t)-F_{\phi_{k_{\ell}}}(u),n_{\epi(\phi_{k_{\ell}})}\left(F_{\phi_{k_{\ell}}}(u)\right)\right\rangle \le0.
\]
 Letting $\ell\to\infty$ we have $\left\langle (x,t)-F_{\phi}(u),\widetilde{n}\right\rangle \le0$,
so $\widetilde{n}=n_{\epi(\phi)}\left(F_{\phi}(u)\right)$ since we
assumed the outer unit normal is unique. This proves the claim. 

Now the convergence $\nabla s_{\phi_{k}}(u)\to\nabla s_{\phi}(u)$
follows immediately from the explicit formula (\ref{eq:radial-grad}).
A direct differentiation shows that 
\begin{equation}
DF_{\phi}(u)=\left(\begin{array}{c}
\begin{array}{c}
\\\nabla s_{\phi}(u)\otimes u+s_{\phi}(u)\cdot I_{n}\\
\\\end{array}\\
\hline -\nabla s_{\phi}(u)/s_{\phi}(u)
\end{array}\right)\label{eq:DF-formula}
\end{equation}
 from which it is also clear that $DF_{\phi_{i}}(u)\to DF_{\phi}(u)$. 

\end{enumerate}
The last ingredient we need for our proof is a bound on the curvilinear
radial function of a simple convex function:
\end{proof}
\begin{lem}
\label{lem:radial-bound}Define $\psi\in\cvx$ by $\psi(x)=a\left|x\right|+b$
for $a>0$ and $b\in\RR$. Then $s_{\psi}(u)\le C\frac{\log(1+\left|u\right|)}{\left|u\right|}$
for all $u\in\RR^{n}$, where $C>0$ depends only on $a$ and $b$. 
\end{lem}

\begin{proof}
We choose $C=\max\left\{ \frac{1}{a},e^{-b}\right\} $. Plugging $a=-\frac{\left|u\right|}{1+\left|u\right|}$
into the standard inequality $\log(1+a)\le a$ we see that $\log\left(1+\left|u\right|\right)\ge\frac{\left|u\right|}{1+\left|u\right|}$.
Therefore for $\sigma=C\frac{\log(1+\left|u\right|)}{\left|u\right|}$
we have 
\begin{align*}
\log\left(\frac{1}{\sigma}\right) & \le\log\left(\frac{1+\left|u\right|}{C}\right)=\log\left(1+\left|u\right|\right)-\log C\\
 & \le aC\log(1+\left|u\right|)+b=\psi(\sigma u).
\end{align*}
 Therefore $s_{\psi}(u)\le\sigma$ as claimed. 
\end{proof}
With all the ingredient in place we are finally ready to prove Theorem
\ref{thm:measures-cont}:
\begin{proof}[Proof of Theorem \ref{thm:measures-cont}.]

Fix $\left\{ f_{k}\right\} _{k=1}^{\infty},f\in\lc$ such that $f_{k}\to f$.
Write as usual $f_{k}=e^{-\phi_{k}}$ and $f=e^{-\phi}$. Our goal
is to prove that for all cosmically continuous functions $\xi:\RR^{n}\to\RR$
we have 
\[
\int_{\RR^{n}}\xi\dd\mu_{f_{k}}+\int_{\RR^{n}}\overline{\xi}\dd\mu_{f_{k}}\to\int_{\RR^{n}}\xi\dd\mu_{f}+\int_{\RR^{n}}\overline{\xi}\dd\mu_{f},
\]
which by (\ref{eq:var-alternative}) is equivalent to 
\[
\int_{\partial\epi(\phi_{k})}\widehat{\xi}\left(n_{\epi(\phi_{k})}(x,t)\right)e^{-t}\dd\HH^{n}(x,t)\to\int_{\partial\epi(\phi)}\widehat{\xi}\left(n_{\epi(\phi)}(x,t)\right)e^{-t}\dd\HH^{n}(x,t).
\]

Using Lemma \ref{lem:cvx-eps-properties}(\ref{enu:cvx-eps-trans}),
we may translate all functions by the same vector $v$ and assume
that $\phi\in\cvx^{(\epsilon)}$ for some $0<\epsilon<\frac{1}{2}$.
This does not change the measures $\left(\mu_{f_{k}},\nu_{f_{k}}\right)$
or $\left(\mu_{f},\nu_{f}\right)$. By part (\ref{enu:cvx-eps-open})
of the same lemma we know that $\phi_{k}\in\cvx^{(\epsilon/2)}$ for
all large enough $k$. We may therefore perform the change of variables
$(x,t)=F_{\phi_{k}}(u)$ on the left hand side and $(x,t)=F_{\phi}(u)$
on the right hand side, and conclude that our goal is equivalent to
\[
\int_{\RR^{n}}\widehat{\xi}\left(n_{\epi(\phi_{k})}(F_{\phi_{k}}(u))\right)s_{\phi_{k}}(u)\left|JF_{\phi_{k}}(u)\right|\dd u\to\int_{\RR^{n}}\widehat{\xi}\left(n_{\epi(\phi)}(F_{\phi}(u))\right)s_{\phi}(u)\left|JF_{\phi}(u)\right|\dd u.
\]
Here of course $JF_{\phi}$ denotes the Jacobian of the map $F_{\phi}$.

By (\ref{eq:normal-converge}) we know that $n_{\epi(\phi_{k})}(F_{\phi_{k}}(u))\to n_{\epi(\phi)}(F_{\phi}(u))$
almost everywhere, and since $\widehat{\xi}$ is continuous $\widehat{\xi}\left(n_{\epi(\phi_{k})}(F_{\phi_{k}}(u))\right)\to\widehat{\xi}\left(n_{\epi(\phi)}(F_{\phi}(u))\right)$
almost everywhere. By Proposition \ref{prop:radial-continuous} we
also know that $s_{\phi_{k}}(u)\to s_{\phi}(u)$ and 
\[
JF_{\phi_{k}}(u)=\sqrt{\det\left(\left(DF_{\phi_{k}}(u)\right)^{t}DF_{\phi_{k}}(u)\right)}\to\sqrt{\det\left(\left(DF_{\phi}(u)\right)^{t}DF_{\phi}(u)\right)}=JF_{\phi}(u)
\]
almost everywhere. Therefore in order to conclude the proof it is
enough to justify the use of the dominated convergence theorem. 

Since $\phi_{k}\to\phi$, there exists $\psi(x)=a\left|x\right|+b$
such that $\phi_{k}\ge\psi$ for all $k$, and then by Lemma \ref{lem:radial-bound}
\[
s_{\phi_{k}}(u)\le s_{\psi}(u)\le C\frac{\log(1+\left|u\right|)}{\left|u\right|}
\]
 for all $u\in\RR^{n}$. Here and after we use $C>0$ to denote some
constant that does not depend on $k$ or $u$, whose exact value may
change from line to line. 

We saw in (\ref{eq:DF-formula}) that $DF_{\phi_{k}}$ is an $(n+1)\times n$
matrix of the form $DF_{\phi}(u)=\left(\begin{array}{c}
\begin{array}{c}
A\end{array}\\
\hline w
\end{array}\right)$ where $A=\nabla s_{\phi}(u)\otimes u+s_{\phi}(u)\cdot I_{n}$ and
$w=-\nabla s_{\phi}(u)/s_{\phi}(u)$. Therefore 
\begin{align}
JF_{\phi_{k}}(u) & =\sqrt{\det\left(\left(DF_{\phi_{k}}(u)\right)^{t}DF_{\phi_{k}}(u)\right)}=\sqrt{\det\left(A^{t}A+w\otimes w\right)}\le\left\Vert A^{t}A+w\otimes w\right\Vert ^{\frac{n}{2}}\label{eq:JF-bound}\\
 & \le\left(\left\Vert A\right\Vert ^{2}+\left|w\right|^{2}\right)^{\frac{n}{2}}\le\left(\left(\left|\nabla s_{\phi}(u)\right|\left|u\right|+s_{\phi}(u)\right)^{2}+\left|w\right|^{2}\right)^{\frac{n}{2}}.\nonumber 
\end{align}
 Here $\left\Vert \cdot\right\Vert $ denotes the operator norm, or
the largest singular value. 

Using the explicit formula (\ref{eq:radial-grad}) and Lemma \ref{lem:cvx-eps-properties}(\ref{enu:cvx-eps-bound})
we see that 
\[
\left|w\right|=\frac{\left|\nabla s_{\phi_{k}}(u)\right|}{s_{\phi_{k}}(u)}=s_{\phi_{k}}(u)\cdot\frac{\left|n_{\epi(\phi_{k})}^{(x)}(x,t)\right|}{\left\langle n_{\epi(\phi_{k})}(x,t),(x,-1)\right\rangle }\le C\frac{\log(1+\left|u\right|)}{\left|u\right|}\cdot\frac{1}{\epsilon/2}=C\frac{\log(1+\left|u\right|)}{\left|u\right|}.
\]
Similarly 
\[
\left|\nabla s_{\phi_{k}}(u)\right|\left|u\right|\le C\left(\frac{\log(1+\left|u\right|)}{\left|u\right|}\right)^{2}\left|u\right|\le C\frac{\log^{2}(1+\left|u\right|)}{\left|u\right|}.
\]
 Plugging these estimates into (\ref{eq:JF-bound}) and using the
fact that $\widehat{\xi}$ is bounded on $\overline{\SS_{-}^{n}}$
we obtain 
\begin{align*}
\left|\widehat{\xi}\left(n_{\epi(\phi_{k})}(F_{\phi_{k}}(u))\right)s_{\phi_{k}}(u)\left|JF_{\phi_{k}}(u)\right|\right| & \le C\frac{\log\left(1+\left|u\right|\right)}{\left|u\right|}\cdot\left(\frac{\log^{2}(1+\left|u\right|)+\log\left(1+\left|u\right|\right)}{\left|u\right|}\right)^{n}\\
 & \le C\frac{\max\left\{ \log^{2n+1}(1+\left|u\right|),\log^{n+1}(1+\left|u\right|)\right\} }{\left|u\right|^{n+1}.}
\end{align*}
 As this function is integrable on $\RR^{n}$ the use of dominated
convergence is justified, and the proof is complete. 
\end{proof}

\section{\label{sec:reverse-continuity}Continuity of solution to the Minkowski
problem}

In this section we prove Theorem \ref{thm:functions-cont}. We begin
by collecting some preliminary facts. First, we will need the following
result of Colesanti and Fragalà:
\begin{prop}[{\cite[Proposition 3.11]{ColesantiFragalaFirst2013}}]
\label{prop:d(f,f)}For all $f\in\lc$ we have 
\[
\delta(f,f)=n\int f+\int f\log f.
\]

\end{prop}

Next, we need an isoperimetric inequality for log-concave functions,
first shown in \cite{MilmanRotemMixed2013}:
\begin{prop}[{\cite[Proposition 27]{MilmanRotemMixed2013}}]
\label{prop:isoperimetric}For all $f\in\lc$ we have
\[
\delta(f,\oo_{B_{2}^{n}})\ge c_{n}\left(\max f\right)^{\frac{1}{n}}\left(\int f\right)^{1-\frac{1}{n}},
\]
for some constant $c_{n}>0$ that depends only on the dimension $n$. 
\end{prop}

In fact, it was shown in \cite{MilmanRotemMixed2013} that optimal
value of the constant $c_{n}$ is attained when $f(x)=e^{-\left|x\right|}$.
For us however the exact value of $c_{n}$ and the nature of the minimizers
will not play any role. 

Finally, we need the following strengthening of Lemma \ref{lem:functional-bound}(\ref{enu:norm-comparison-one}):
\begin{lem}
\label{lem:integral-lower-bound}Fix $\left\{ f_{k}\right\} _{k=1}^{\infty},f\in\lc$
and assume that $\left(\mu_{f_{k}},\nu_{f_{k}}\right)\to\left(\mu_{f},\nu_{f}\right)$
cosmically. Then there exists a constant $c>0$, such that for all
$y\in\RR^{n}$ and all $k$ we have
\[
\int_{\RR^{n}}\left|\left\langle x,y\right\rangle \right|\dd\mu_{f_{k}}(x)+\int_{\SS^{n-1}}\left|\left\langle x,y\right\rangle \right|\dd\nu_{f_{k}}(x)\ge c\left|y\right|.
\]
\end{lem}

Of course, the main point of this strengthening is that the constant
$c$ is not allowed to depend on $k$. 
\begin{proof}
Define a sequence of functions $\Phi_{k}:\RR^{n}\to\RR$ by 
\[
\Phi_{k}(y)=\int_{\RR^{n}}\left|\left\langle x,y\right\rangle \right|\dd\mu_{f_{k}}(x)+\int_{\SS^{n-1}}\left|\left\langle x,y\right\rangle \right|\dd\nu_{f_{k}}(x),
\]
 and similarly define 
\[
\Phi(y)=\int_{\RR^{n}}\left|\left\langle x,y\right\rangle \right|\dd\mu_{f_{}}(x)+\int_{\SS^{n-1}}\left|\left\langle x,y\right\rangle \right|\dd\nu_{f}(x).
\]

By Lemma \ref{lem:functional-bound}, there exist positive constants
$\left\{ c_{k}\right\} _{k=0}^{\infty}$ such that for all $y\in\RR^{n}$
we have $\Phi(y)\ge c_{0}\left|y\right|$ and $\Phi_{k}(y)\ge c_{k}\left|y\right|$.
Since the function $\ell_{y}(x)=\left|\left\langle x,y\right\rangle \right|$
is cosmically continuous and $\overline{\ell_{y}}=\ell_{y}$, it follows
from the definition of cosmic convergence that $\Phi_{k}(y)\to\Phi(y)$
for all $y\in\RR^{n}$. Since the functions $\left\{ \Phi_{k}\right\} _{k=1}^{\infty}$
are convex and finite, the pointwise convergence $\Phi_{k}\to\Phi$
implies uniform convergence on compact subsets of $\RR^{n}$ (see
e.g. \cite[Theorem 7.17(c)]{RockafellarWetsVariational1998}). In
particular 
\[
\min_{y\in\SS^{n-1}}\Phi_{k}(y)\to\min_{y\in\SS^{n-1}}\Phi(y)\ge c_{0}.
\]
 Therefore there exists $K>0$ such that for all $k>K$ we have and
all $y\in\SS^{n-1}$ we have $\Phi_{k}(y)\ge\frac{c_{0}}{2}$. It
is therefore enough to choose $c=\min\left\{ \frac{c_{0}}{2},c_{1},c_{2},\ldots,c_{K}\right\} $. 
\end{proof}
We are now ready to prove Theorem \ref{thm:functions-cont}:
\begin{proof}[Proof of Theorem \ref{thm:functions-cont}.]
 Assume by contradiction that we are given $\left\{ f_{k}\right\} _{k=1}^{\infty},f\in\lc$
such that $\left(\mu_{f_{k}},\nu_{f_{k}}\right)\to\left(\mu_{f},\nu_{f}\right)$
cosmically but $f_{k}\not\to f$, even up to translations. Fix any
metric $d$ on $\lc$ which generates our notion of convergence (examples
to such metrics were given in \cite{LiMussnigMetrics2022}), and define
\[
\widetilde{d}(f,g)=\inf_{v\in\RR^{n}}d(f,\tau_{v}g)
\]
 where $\tau_{v}g(x)=g(x+v)$. Our assumption precisely means that
$\widetilde{d}(f_{k},f)\not\to0$. By passing to a sub-sequence we
may assume without loss of generality that there exists an $\epsilon>0$
such that $\widetilde{d}(f_{k},f)\ge\epsilon$ for all $k$. 

Write as usual $f_{k}=e^{-\phi_{k}}$ and $f=e^{-\phi}$. Since we
allow translations, we may assume without loss of generality that
$\min\phi_{k}=\phi_{k}(0)$ for all $k$. Combining Lemma \ref{lem:integral-lower-bound}
with the second part of Lemma \ref{lem:functional-bound} we see that
there exists a constant $c>0$ such for every $k$, every $\psi\in\cvx$
with $\min\psi=\psi(0)$ and every $x\in\RR^{n}$ we have 
\[
\psi(x)\ge\frac{1}{\mu_{f_{k}}(\RR^{n})}\left(c\left|x\right|-\int_{\RR^{n}}\psi^{\ast}\dd\mu_{f_{k}}-\int_{\SS^{n-1}}\overline{\psi^{\ast}}\dd\nu_{f_{k}}\right).
\]
We now choose $\psi=\phi_{k}$ and use the first variation formula
\ref{eq:representation} and Proposition \ref{prop:d(f,f)} to conclude
that 
\begin{align}
\phi_{k}(x) & \ge\frac{1}{\mu_{f_{k}}(\RR^{n})}\left(c\left|x\right|-\int_{\RR^{n}}\phi_{k}^{\ast}\dd\mu_{f_{k}}-\int_{\SS^{n-1}}\overline{\phi_{k}^{\ast}}\dd\nu_{f_{k}}\right)=\frac{1}{\mu_{f_{k}}(\RR^{n})}\left(c\left|x\right|-\delta(f_{k},f_{k})\right)\nonumber \\
 & =\frac{1}{\mu_{f_{k}}(\RR^{n})}\left(c\left|x\right|-n\int f_{k}-\int f_{k}\log f_{k}\right)=\frac{c}{\int f_{k}}\left|x\right|-n-\frac{\int f_{k}\log f_{k}}{\int f_{k}}.\label{eq:uniform-coercive}
\end{align}
 Note that we also used the fact that $\mu_{f_{k}}\left(\RR^{n}\right)=\int f_{k}$
which is obvious from the definition of $\mu_{f_{k}}$. 

Using the fact that the constant function $\oo$ is cosmically continuous
we obtain
\[
\int f_{k}=\mu_{f_{k}}(\RR^{n})=\int\oo\dd\mu_{f_{k}}+\int\overline{\oo}\dd\nu_{f_{k}}\to\int\oo\dd\mu_{f}+\int\overline{\oo}\dd\nu_{f}=\mu_{f}(\RR^{n})=\int f,
\]
 so in particular $\left\{ \int f_{k}\right\} _{k=1}^{\infty}$ is
bounded from above and from below by some positive constants. Similarly
using the fact that $x\mapsto\left|x\right|$ is cosmically continuous
and (\ref{eq:representation}) we have 
\[
\delta(f_{k},\oo_{B_{2}^{n}})=\int\left|x\right|\dd\mu_{f_{k}}+\int\left|x\right|\dd\nu_{f_{k}}\to\int\left|x\right|\dd\mu_{f}+\int\left|x\right|\dd\nu=\delta(f,\oo_{B_{2}^{n}}),
\]
 so the sequence $\left\{ \delta(f_{k},\oo_{B_{2}^{n}})\right\} _{k=1}^{\infty}$
is also bounded. Using Proposition \ref{prop:isoperimetric} we have
\[
\max f_{k}\le C_{n}\frac{\delta(f_{k},\oo_{B_{2}^{n}})^{n}}{\left(\int f_{k}\right)^{n-1}},
\]
 so the sequence $\left\{ \max f_{k}\right\} _{k=1}^{\infty}$ is
also bounded from above. Therefore 
\[
\frac{\int f_{k}\log f_{k}}{\int f_{k}}\le\frac{\max\left(\log f_{k}\right)\int f_{k}}{\int f_{k}}=\log\left(\max f_{k}\right)
\]
 is also bounded from above. Plugging all these estimates into (\ref{eq:uniform-coercive})
we conclude that the sequence $\left\{ \phi_{k}\right\} _{k=1}^{\infty}$
is uniformly coercive. Since $\left\{ \int e^{-\phi_{k}}\right\} _{k=1}^{\infty}=\left\{ \int f_{k}\right\} _{k=1}^{\infty}$
is bounded from below by a positive constant, the same argument as
in Proposition \ref{prop:minimizer-exists} shows that $\sup_{k}\left(\min\phi_{k}\right)<\infty$. 

Theorem \ref{thm:selection} now allows us to pass to another sub-sequence
and assume without loss of generality that $f_{k}\to g$ for some
$g\in\lc$ (the fact that $\int g>0$ was not part of the theorem,
but it follows from the fact that $\int g=\lim_{k\to\infty}\int f_{k}$).
By Theorem \ref{thm:measures-cont} it follows that $\left(\mu_{f_{k},}\nu_{f_{k}}\right)\to(\mu_{g},\nu_{g})$
cosmically. But we also have $\left(\mu_{f_{k},}\nu_{f_{k}}\right)\to(\mu_{f},\nu_{f})$,
and since the cosmic limit is clearly unique we have $(\mu_{g},\nu_{g})=(\mu_{f},\nu_{f})$.
But then by the uniqueness part of Theorem \ref{thm:main-Minkowksi}
we must have $g=\tau_{v}f$ for some $v\in\RR^{n}$. This implies
that 
\[
\widetilde{d}(f_{k},f)=\widetilde{d}(f_{k},g)\le d(f_{k},g)\to0,
\]
 contradicting our assumption that $\widetilde{d}(f_{k},f)\ge\epsilon$
for all $k$. 
\end{proof}
\bibliographystyle{plain}
\bibliography{/Users/lrotem/Library/CloudStorage/OneDrive-Technion/citations}

\end{document}